\documentclass[a4paper]{amsart}

\usepackage{amsmath,amstext,amssymb,mathrsfs,amscd,indentfirst, bbm, mathtools}
\usepackage{amsthm}
\usepackage{amsfonts}
\usepackage{microtype}
\usepackage[dvipsnames,svgnames,x11names,hyperref]{xcolor}
\usepackage{stmaryrd}

\usepackage{booktabs}
\usepackage{verbatim}
\usepackage{enumerate}

\usepackage{graphicx}

\usepackage{tikz, tikz-cd}
\usetikzlibrary{matrix,calc,positioning,arrows,decorations.pathreplacing,decorations.markings,patterns}
\tikzset{->-/.style={decoration={
			markings,
			mark=at position #1 with {\arrow{>}}},postaction={decorate}}}
\tikzset{-<-/.style={decoration={
					markings,
					mark=at position #1 with {\arrow{<}}},postaction={decorate}}}

\newcommand{\bC}{\mathbb{C}}

\newcommand{\bF}{\mathbb{F}}

\newcommand{\bQ}{\mathbb{Q}}
\newcommand{\bR}{\mathbb{R}}

\newcommand{\bZ}{\mathbb{Z}}

\newcommand{\cL}{\mathcal{L}}

\newcommand{\cV}{\mathcal{V}}
\newcommand{\cW}{\mathcal{W}}

\newcommand\lra{\longrightarrow}

\newcommand\Diff{\mathrm{Diff}}

\newcommand\colim{\operatorname*{colim}}

\newcommand{\hcoker}{/\!\!/}

\newcommand{\map}{\mathrm{map}}

\newcommand{\Tor}{\mathrm{Tor}}

\renewcommand{\epsilon}{\varepsilon}

\newcommand{\SO}{\mathrm{SO}}
\newcommand{\OO}{\mathrm{O}}

\newcommand{\cH}{\mathcal{H}}

\newcommand{\sign}{\mathrm{sign}}

\mathchardef\ordinarycolon\mathcode`\:
\mathcode`\:=\string"8000
\begingroup \catcode`\:=\active
  \gdef:{\mathrel{\mathop\ordinarycolon}}
\endgroup

\theoremstyle{plain}
\newtheorem{MainThm}{Theorem}

\newtheorem{theorem}{Theorem}[section]

\newtheorem{proposition}[theorem]{Proposition}
\newtheorem{lemma}[theorem]{Lemma}

\newtheorem{corollary}[theorem]{Corollary}

\theoremstyle{definition}
\newtheorem{definition}[theorem]{Definition}

\newtheorem{example}[theorem]{Example}

\theoremstyle{remark}

\newtheorem{remark}[theorem]{Remark}
\newtheorem*{remark*}{Remark}
\newtheorem*{convention*}{Convention}

\numberwithin{equation}{section}

\usepackage[hypertexnames=false]{hyperref}

\title{On the cohomology of Torelli groups. II}

\author{Oscar Randal-Williams}
\email{o.randal-williams@dpmms.cam.ac.uk}
\address{Centre for Mathematical Sciences\\
Wilberforce Road\\
Cambridge CB3 0WB\\
UK}

\begin{document}
\begin{abstract}
We describe the ring structure of the rational cohomology of the Torelli groups of the manifolds $\#^g S^n \times S^n$ in a stable range, for $2n \geq 6$. Some of our results are also valid for $2n=2$, where they are closely related to unpublished results of Kawazumi and Morita.
\end{abstract}
\maketitle

\setcounter{tocdepth}{1}
\tableofcontents

\section{Introduction}

This paper can be considered as a (somewhat extensive) addendum to our earlier work with Kupers \cite{KR-WTorelli}. We shall be concerned with the manifold $W_g := \#^g S^n \times S^n$ generalising to higher dimensions the orientable surface of genus $g$, its topological group $\Diff^+(W_g)$ of orientation-preserving diffeomorphisms, and various subgroups of it. The first kind of subgroups are $\Diff(W_g, D^{2n}) \leq \Diff^+(W_g, *) \leq \Diff^+(W_g)$, the diffeomorphisms which fix a disc and a point respectively. The second kind are their \emph{Torelli subgroups}
$$\Tor(W_g, D^{2n}), \quad \Tor^+(W_g, *), \quad \Tor^+(W_g),$$
consisting of those diffeomorphisms which in addition act trivially on $H_n(W_g;\bZ)$. The intersection form on this middle cohomology group is nondegenerate and $(-1)^n$-symmetric, giving a homomorphism
$$\alpha_g : \mathrm{Diff}^{+}(W_g) \lra G_g := \begin{cases}
\mathrm{Sp}_{2g}(\bZ) & \text{ if $n$ is odd,}\\
\mathrm{O}_{g,g}(\bZ) &  \text{ if $n$ is even.}
\end{cases}$$
This map is not always surjective, but its image is a certain finite-index subgroup $G'_g \leq G_g$, even when restricted to $\Diff(W_g, D^{2n})$, so there is an outer $G'_g$-action on each of the Torelli subgroups. This makes the rational cohomology of each of the Torelli groups into $G'_g$-representations.

In \cite{KR-WTorelli}, for $2n \geq 6$ we determined $H^*(B\Tor(W_g, D^{2n});\bQ)$ as a ring and as a $G'_g$-representation in a range of degrees tending to infinity with $g$. Using the Serre spectral sequence associated to various simple fibrations relating the different Torelli groups we were able to also determine $H^*(B\Tor^+(W_g, *);\bQ)$ and $H^*(B\Tor^+(W_g);\bQ)$ as $G'_g$-representations. This kind of argument was not able to determine the ring structure, however, as multiplicative information gets lost when passing to the associated graded of the Serre filtration. Here we shall determine $H^*(B\Tor^+(W_g, *);\bQ)$ and $H^*(B\Tor^+(W_g);\bQ)$ as $\bQ$-algebras too: this is achieved in Theorem \ref{thm:TorelliCalc}. The statement given there is more powerful, but just as in \cite[Section 5]{KR-WTorelli} one can extract from it the following presentation for $H^*(B\Tor^+(W_g);\bQ)$, which is easier to parse. (A presentation for $H^*(B\Tor^+(W_g, *);\bQ)$ can be extracted in a similar way.)

Let us write $H(g) := H^n(W_g;\bQ)$, on which $G'_g$ operates in the evident way. Let $\lambda : H(g) \otimes H(g) \to \bQ$ denote the intersection form, and $\{a_i\}_{i=1}^{2g}$ be a basis of $H(g)$ with dual basis $\{a_i^\#\}_{i=1}^{2g}$ in the sense that $\lambda(a_i^\#, a_j) = \delta_{ij}$, so that the form dual to the pairing $\lambda$ is $\omega = \sum_{i=1}^{2g} a_i \otimes a_i^\#$. In Section \ref{sec:ModifiedMMM} we will construct certain ``modified twisted Miller--Morita--Mumford classes'', which when restricted to the Torelli group yield $G'_g$-equivariant maps
$$\bar{\kappa}_c : H(g)^{\otimes r} \lra H^{n(r-2) + |c|}(B\Tor^+(W_g);\bQ)$$
for each $c \in \bQ[e, p_1, p_2, \ldots, p_{n-1}] = H^*(B\SO(2n);\bQ)$ and each $s \geq 0$. When $r=0$ we write $\bar{\kappa}_c = \bar{\kappa}_c(1)$; these agree with the usual Miller--Morita--Mumford classes $\kappa_c$. Finally, we write $\chi := \chi(W_g) = 2 + (-1)^n 2g$ and will always suppose that this is not zero (i.e.\ that $(n,g) \neq (\text{odd}, 1)$).

\begin{MainThm}\label{thm:A}
If $2n \geq 6$ then, in a range of degrees tending to infinity with $g$, $H^{*}(B\Tor^+(W_g);\bQ)$ is generated as a $\bQ$-algebra by the classes $\bar{\kappa}_c(v_1 \otimes \cdots \otimes v_r)$ for $c$ a monomial in $e, p_1, \ldots, p_{n-1}$, and $r \geq 0$, such that $n(r-2) + |c| > 0$. A complete set of relations in this range is given by
\begin{enumerate}[(i)]
\item $\bar{\kappa}_c(v_{\sigma(1)} \otimes \cdots \otimes v_{\sigma(r)}) = \sign(\sigma)^n \cdot \bar{\kappa}_c(v_1 \otimes \cdots \otimes v_r)$,

\item $\bar{\kappa}_e(v_1) = 0$,

\item $\begin{aligned}[t]
\sum_i \bar{\kappa}_x(v \otimes a_i) \cdot \bar{\kappa}_y(a_i^\# \otimes w) &= \bar{\kappa}_{x\cdot y}(v \otimes w) + \tfrac{1}{\chi^2}\bar{\kappa}_{e^2} \cdot \bar{\kappa}_x(v) \cdot \bar{\kappa}_y(w)\\
&\nonumber \quad\quad - \tfrac{1}{\chi}\bigl(\bar{\kappa}_{e \cdot x}(v) \cdot \bar{\kappa}_y(w) + \bar{\kappa}_{x}(v) \cdot \bar{\kappa}_{e \cdot y}(w)\bigr),
\end{aligned}$

\item $\begin{aligned}[t]\sum_i \bar{\kappa}_x(v \otimes a_i \otimes a_i^\#) = \tfrac{\chi-2}{\chi} \bar{\kappa}_{e\cdot x}(v) + \tfrac{1}{\chi^2}\bar{\kappa}_{e^2} \cdot \bar{\kappa}_{x}(v)\end{aligned}$,

\item $\bar{\kappa}_{\cL_i} =0$, where $\cL_i$ denotes the $i$th Hirzebruch $\cL$-class,
\end{enumerate}
for $v \in H(g)^{\otimes r}$ and $w \in H(g)^{\otimes s}$.
\end{MainThm}

As in \cite[Section 5.4]{KR-WTorelli} this is not the smallest possible presentation: it can be simplified by manipulating graphs; we leave the details to the interested reader.

For $2n=4$ or $2n=2$ there is still a map from the $\bQ$-algebra given by this presentation to $H^{*}(B\Tor^+(W_g);\bQ)$. If $2n=2$ then (in a stable range) this map is an isomorphism onto the maximal algebraic subrepresentation in degrees $\leq N$, assuming that $H^{*}(B\Tor^+(W_g);\bQ)$ is finite-dimensional in degrees $< N$ for all large enough $g$. This is known to hold for $N=2$ by work of Johnson \cite{JohnsonIII}.

\subsection{Outline}
The overall strategy is parallel to \cite{KR-WTorelli}. There we defined certain twisted Miller--Morita--Mumford classes and used them to describe the twisted cohomology groups $H^*(B\Diff^{+}(W_g, D^{2n}) ; \mathcal{H}^{\otimes s})$ in a stable range of degrees, where $\mathcal{H}$ is the local coefficient system corresponding to $H_n(W_g; \bQ)$ with the action by diffeomorphisms of $W_g$. This calculation was valid for $2n=2$ as well. Using that for $2n \geq 6$ the $G'_g$-representations $H^*(B\Tor^+(W_g, D^{2n});\bQ)$ extend to representations of the ambient algebraic group (namely $\mathrm{Sp}_{2g}$ or $\mathrm{O}_{g,g}$) by \cite{KR-WAlg}\footnote{In fact we did something more complicated in \cite{KR-WTorelli} because this algebraicity result was not known at the time, but please allow some narrative leeway.}, the argument was completed by establishing the degeneration of the Serre spectral sequence
$$E_2^{p,q} = H^p(G'_g ; H^q(B\Tor^+(W_g, D^{2n});\bQ) \otimes \mathcal{H}^{\otimes s}) \Longrightarrow H^{p+q}(B\Diff^{+}(W_g, D^{2n}) ; \mathcal{H}^{\otimes s})$$
using work of Borel, and then using a categorical form of Schur--Weyl duality to extract the structure of $H^*(B\Tor^+(W_g, D^{2n});\bQ)$ from the $H^{*}(B\Diff^{+}(W_g, D^{2n}) ; \mathcal{H}^{\otimes s})$ for all $s$'s and various structure maps between them.

The twisted Miller--Morita--Mumford classes may be defined on $B\Diff^{+}(W_g, *)$ too, but not on $B\Diff^{+}(W_g)$. Our first task will be to define so-called ``modified twisted Miller--Morita--Mumford classes'' in $H^{*}(B\Diff^{+}(W_g) ; \mathcal{H}^{\otimes s})$ and analyse their behaviour: it turns out that their behaviour is significantly more complicated than the unmodified version, though still understandable. We will then use them to describe the twisted cohomology groups $H^*(B\Diff^{+}(W_g) ; \mathcal{H}^{\otimes s})$ in a stable range of degrees. This description will be in terms of a certain vector space of graphs with vertices labelled by monomials in Euler and Pontrjagin classes, which play the role here of the vector spaces of labelled partitions from \cite{KR-WTorelli}. The passage from this calculation to $H^*(B\Tor^{+}(W_g);\bQ)$ is as above.

The case of dimension $2n=2$ is somewhat special, in precisely the same way as it was in \cite{KR-WTorelli}: the calculation of $H^*(B\Diff^{+}(W_g) ; \mathcal{H}^{\otimes s})$ is valid in this case, but as the cohomology of $B\Tor^+(W_g)$ is not even known to be finite-dimensional in a stable range, we cannot make a conclusion about it. (Instead one can make a conclusion about the continuous cohomology of the Torelli group, i.e.\ the Lie algebra cohomology of its Mal'cev Lie algebra: see \cite{KR-WKoszul}, \cite{FNW}, \cite{HainJohnson}.)  In addition, in this case our modified twisted Miller--Morita--Mumford classes are essentially the same as those that have been defined by Kawazumi and Morita \cite{MoritaLinearRep, KM, KMunpub}, and the graphical calculus that we employ is similar to theirs. In Section \ref{sec:surfaces} we fully explain this connection, and also relate it to work of Garoufalidis and Nakamura \cite{GN, GNCorr} and Akazawa \cite{Akazawa}.

To avoid a great deal of repetition we have refrained from spelling out a lot of the background that was given in \cite{KR-WTorelli}, and from giving in detail arguments that are very similar to those given there. As such this paper should not be considered as attempting to be self-contained: given that its interest will be to readers of \cite{KR-WTorelli} this should not present a problem.

\subsection{Acknowledgements} I am grateful to Alexander Kupers for feedback on an earlier draft. I was supported by the ERC under the European Union’s Horizon 2020 research and innovation programme (grant agreement No.\ 756444) and by a Philip Leverhulme Prize from the Leverhulme Trust.

\section{Characteristic classes}

\subsection{Recollection on twisted Miller--Morita--Mumford classes.} 

If $\pi' : E' \to X'$ is an oriented smooth $W_g^{2n}$-bundle equipped with a section $s : X' \to E'$, and $\cH$ denotes the local coefficient system $x \mapsto H_n((\pi')^{-1}(x);\bQ)$ on $X'$, then it is explained in \cite[Section 3.2]{KR-WTorelli} that there is a unique class $\epsilon = \epsilon_s \in H^n(E' ; \cH)$ characterised by
\begin{enumerate}[(i)]
\item for each $x \in X'$ the element $\epsilon\vert_{(\pi')^{-1}(x)} \in H^n((\pi')^{-1}(x);\bQ) \otimes H_n((\pi')^{-1}(x);\bQ)$ is coevaluation, and

\item $s^* \epsilon=0$.
\end{enumerate}
The proof is as follows. The Serre spectral sequence yields an exact sequence
$$0\to H^n(X';\cH) \overset{(\pi')^*}\to H^n(E';\cH) \to H^0(X' ; \cH^\vee \otimes \cH) \overset{d_{n+1}}\to H^{n+1}(X';\cH) \overset{(\pi')^*}\to H^{n+1}(E';\cH)$$
and the section $s$ shows that the right-hand map $(\pi')^*$ is injective, so that the map $d_{n+1}$ is zero, and splits the left-hand map $(\pi')^*$. The class $\text{coev} \in H^0(X' ; \cH^\vee \otimes \cH)$ then gives rise to a unique $\epsilon$ satisfying the given properties. 

We then defined the twisted Miller--Morita--Mumford class
\begin{equation}\label{eq:ordinaryMMM}
\kappa_{\epsilon^a c} = \kappa_{\epsilon^a c}(\pi', s) := \pi'_!( \epsilon^a \cdot c(T_{\pi'} E')) \in H^{(a-2)n + |c|}(X';\cH^{\otimes a}).
\end{equation}

\subsection{Modified twisted Miller--Morita--Mumford classes}\label{sec:ModifiedMMM} 

If $\pi: E \to X$ is an oriented smooth $W_g^{2n}$-bundle but is not equipped with a section then, as long as $\chi := \chi(W_g) = 2 + (-1)^n 2g \neq 0$ (i.e.\ $(n,g) \neq (\text{odd}, 1)$, cf.\ Remark \ref{rem:ChiIsZero}), the cohomological role of the section can  instead be played by the \emph{transfer map}
$$\tfrac{1}{\chi}\pi_!(e \cdot -) : H^*(E;\cH) \lra H^*(X;\cH),$$
where $e := e(T_\pi E) \in H^{2n}(E;\bQ)$ denotes the Euler class of the vertical tangent bundle. The projection formula
$$\tfrac{1}{\chi}\pi_!(e \cdot \pi^*(x)) = \tfrac{1}{\chi}\pi_!(e) \cdot x = \tfrac{\chi}{\chi} x = x$$
shows that this map splits $\pi^*$. Thus in this situation there is a unique class $\bar{\epsilon} \in H^n(E ; \cH)$ characterised by
\begin{enumerate}[(i)]
\item for each $x \in X$ the element $\bar{\epsilon}\vert_{\pi^{-1}(x)} \in H^n(\pi^{-1}(x);\bQ) \otimes H_n(\pi^{-1}(x);\bQ)$ is coevaluation, and

\item $\tfrac{1}{\chi}\pi_!(e \cdot \bar{\epsilon})=0$.
\end{enumerate}
\begin{remark}\label{rem:ChiIsZero}
If $(n,g)=(\text{odd}, 1)$ then there is no class $\bar{\epsilon} \in H^n(E;\mathcal{H})$ satisfying (i) and natural under pullback. To see this it suffices to give one example of  a smooth oriented $W_1$-bundle for which $\bar{\epsilon}$ does not exist. Consider the Borel construction for the evident action of $S^1 \times S^1$ on $W_1 = S^n \times S^n$ given by considering $S^n$ as the unit sphere in $\bC^{(n+1)/2}$. This gives a smoth oriented $W_1$-bundle over $B(S^1 \times S^1)$ with total space $E \simeq \mathbb{CP}^{(n-1)/2} \times \mathbb{CP}^{(n-1)/2}$. Thus $H^n(E;\mathcal{H})=0$ as $n$ is odd but $E$ has a cell structure with only even-dimensional cells.
\end{remark}

By analogy with \eqref{eq:ordinaryMMM} we may then define the modified twisted Miller--Morita--Mumford class
\begin{equation}\label{eq:barMMM}
\kappa_{\bar{\epsilon}^a c} = \kappa_{\bar{\epsilon}^a c}(\pi) := \pi_!( \bar{\epsilon}^a \cdot c(T_{\pi} E) ) \in H^{(a-2)n + |c|}(X;\cH^{\otimes a}).
\end{equation}

If $\pi: E \to X$ does have a section $s : X \to E$ then the class $\epsilon \in H^n(E;\cH)$ is also defined, and we may compare it with $\bar{\epsilon}$ as follows:

\begin{lemma}\label{lem:Compare}
If $\pi : E \to X$ has a section then $\bar{\epsilon} = \epsilon - \tfrac{1}{\chi} \pi^*\kappa_{\epsilon e}$.
\end{lemma}
\begin{proof}
The classes $\epsilon, \bar{\epsilon} \in H^n(E;\cH)$ are both defined, and agree when restricted to the fibres of the map $\pi$, so by considering the Serre spectral sequence for $\pi$ we must have $\bar{\epsilon}-\epsilon = \pi^*(x)$ for some class $x \in H^n(X;\cH)$. Applying $\tfrac{1}{\chi}\pi_!(e \cdot -)$ we see that $x =\tfrac{1}{\chi} \pi_!(e \cdot (\bar{\epsilon}-\epsilon))= 0 - \tfrac{1}{\chi} \pi_!(e \cdot \epsilon) = - \tfrac{1}{\chi} \kappa_{\epsilon e}$. (Here we have used, as we often will, the fact that $e$ has even degree to commute it past $\epsilon$.)
\end{proof}

\begin{remark}[Splitting principle]\label{rem:split}
The pullback
\begin{equation}\label{eq:Pullback}
\begin{tikzcd}
E_1 \times_X E_2 \dar{\mathrm{pr}_1} \rar{\mathrm{pr}_2} & E_2 \dar{\pi_2}\\
 E_1 \rar{\pi_1}& X,
\end{tikzcd}
\end{equation}
where $\pi_i : E_i \to X$ are copies of the map $\pi$, is equipped with a section given by the diagonal map $\Delta : E_1 \to E_1 \times_X E_2$. As the maps $\pi_1^*$ and $\mathrm{pr}_2^*$ are injective (they are split by their corresponding transfer maps), for the purpose of establishing identities between the characteristic classes we have discussed it suffices to do so for bundles which do have a section.
\end{remark}

There is another description of $\bar{\epsilon}$ which is sometimes useful. Let $\mathrm{pr}_1 : E \times_X E \to E$ be as in \eqref{eq:Pullback}, which is an oriented $W_g$-bundle with section given by the diagonal map $\Delta$, and so has the class $\kappa_{\epsilon e}(\mathrm{pr}_1, \Delta)$ defined.

\begin{lemma}
We have $\bar{\epsilon} = - \tfrac{1}{\chi} \kappa_{\epsilon e}(\mathrm{pr}_1, \Delta) \in H^n(E;\cH)$.
\end{lemma}
\begin{proof}
By Remark \ref{rem:split} we may suppose without loss of generality that $\pi: E \to X$ has a section $s : X \to E$, defining a class $\epsilon = \epsilon_s \in H^n(E;\cH)$. Consider the pullback square \eqref{eq:Pullback}; let $e_i = (\mathrm{pr}_i)^*(e) \in H^{2n}(E_1 \times_X E_2;\bQ)$ be the Euler class of the vertical tangent bundle on the $i$th factor. Considering $\mathrm{pr}_1$ as a $W_g$-bundle with section given by the diagonal map $\Delta$, there is a class $\epsilon_\Delta \in H^n(E_1 \times_X E_2 ; \cH)$ defined. As both $\epsilon_\Delta$ and $\mathrm{pr}_2^*(\epsilon_s)$ restrict to coevaluation on the fibres of $\mathrm{pr}_1$, we have $\epsilon_\Delta - \mathrm{pr}_2^*(\epsilon_s) = \mathrm{pr}_1^*(x)$ for some class $x \in H^n(E_1;\cH)$. Pulling this equation back along $\Delta$ shows that $x = -\epsilon_s$, so $\epsilon_\Delta = \mathrm{pr}_2^*(\epsilon_s) - \mathrm{pr}_1^*(\epsilon_s)$. Then we have
\begin{align*}
\kappa_{\epsilon e}(\mathrm{pr}_1, \Delta) &= (\mathrm{pr}_1)_!(\epsilon_\Delta \cdot e_2)\\
&= (\mathrm{pr}_1)_!((\mathrm{pr}_2^*(\epsilon_s) - \mathrm{pr}_1^*(\epsilon_s)) \cdot e_2)\\
&= (\mathrm{pr}_1)_!(\mathrm{pr}_2^*(\epsilon_s \cdot e)) - (\mathrm{pr}_1)_!(\mathrm{pr}_1^*(\epsilon_s) \cdot e_2)\\
&= \pi_1^* (\pi_2)_!(\epsilon_s \cdot e) - \chi \epsilon_s\\
&= \pi_1^*\kappa_{\epsilon e}(\pi,s) - \chi \epsilon_s\\
&= - \chi \bar{\epsilon}
\end{align*}
as required.
\end{proof}

The intersection form of the fibres of $\pi : E \to X$ provides a map of local coefficient systems $\lambda : \mathcal{H} \otimes \mathcal{H} \to \bQ$; as we will often be concerned with applying it to two factors of a tensor power $\mathcal{H}^{\otimes k}$ and will have to specify which factors we apply it to, we will denote $\lambda$ by $\lambda_{1,2}$ and more generally write $\lambda_{i,j} : \mathcal{H}^{\otimes k} \to \mathcal{H}^{\otimes k-2}$ for the map that applies $\lambda$ to the $i$th and $j$th factors. We call such operations \emph{contraction}.

If $p : E_1 \times_X E_2 \to X$ denotes the fibre product of two copies of $\pi : E \to X$, and if this has a section $s : X \to E$, then in \cite[Lemma 3.9]{KR-WTorelli} we have established the formula
\begin{equation}\label{eq:Lemma3pt9}
\lambda_{1,2}({\epsilon} \times {\epsilon}) = \Delta_!(1) - 1 \times v - v \times 1 + p^* s^* e \in H^{2n}(E_1 \times_X E_2;\bQ),
\end{equation}
where $v = s_!(1) \in H^{2n}(E;\bQ)$ is the fibrewise Poincar{\'e} dual to the section $s$, cf.\ \cite[Lemma 3.1]{KR-WTorelli}. The analogue of this formula for $\bar{\epsilon}$ is as follows.

\begin{lemma}\label{lem:ModifiedContractionFormula}
We have 
$$\lambda_{1,2}(\bar{\epsilon} \times \bar{\epsilon}) = \Delta_!(1)  + \tfrac{1}{\chi^2} p^*\kappa_{e^2} - \tfrac{1}{\chi}(e \times 1 + 1 \times e) \in H^{2n}(E_1 \times_X E_2 ; \bQ).$$
\end{lemma}

\begin{proof}
As in Remark \ref{rem:split} we may suppose without loss of generality that $\pi: E \to X$ has a section $s : X \to E$, so that $\epsilon \in H^n(E;\cH)$ is defined.

By Lemma \ref{lem:Compare} we have $\bar{\epsilon} = \epsilon - \tfrac{1}{\chi} \pi^*\kappa_{\epsilon e} \in H^n(E;\cH)$, and so
\begin{align*}
\lambda_{1,2}(\bar{\epsilon} \times \bar{\epsilon}) &= \lambda_{1,2}((\epsilon - \tfrac{1}{\chi} \pi^*\kappa_{e \cdot \epsilon}) \times (\epsilon - \tfrac{1}{\chi} \pi^*\kappa_{e \cdot \epsilon}))\\
&= \lambda_{1,2}(\epsilon \times \epsilon) - \lambda_{1,2}(\tfrac{1}{\chi} \pi^*\kappa_{\epsilon e} \times \epsilon)\\
&\quad\quad -\lambda_{1,2}(\epsilon \times \tfrac{1}{\chi} \pi^*\kappa_{\epsilon e}) + \lambda_{1,2}(\tfrac{1}{\chi} \pi^*\kappa_{\epsilon e} \times \tfrac{1}{\chi} \pi^*\kappa_{\epsilon e}).
\end{align*}
The first term is given by \eqref{eq:Lemma3pt9}, and using \cite[Proposition 3.10]{KR-WTorelli} the last term is given by
$$\lambda_{1,2}(\tfrac{1}{\chi} \pi^*\kappa_{\epsilon e} \times  \tfrac{1}{\chi} \pi^*\kappa_{\epsilon  e}) = \tfrac{1}{\chi^2}p^*\lambda_{1,2}(\kappa_{\epsilon  e} \cdot \kappa_{\epsilon  e}) = \tfrac{1}{\chi^2}p^*(\kappa_{e^2} + (\chi^2-2\chi) s^*e ).$$

For the middle two terms, note that 
$$\epsilon \times \tfrac{1}{\chi} \pi^*\kappa_{ \epsilon e} = \tfrac{1}{\chi}(\epsilon \times 1) \cdot p^*(\kappa_{e \cdot \epsilon}) = \tfrac{1}{\chi}(\epsilon \cdot  \pi^*\kappa_{\epsilon e}) \times 1$$
so we need to calculate $\lambda_{1,2}(\epsilon \cdot  \pi^*\kappa_{ \epsilon e}) \in H^{2n}(E;\bQ)$. The class $\epsilon \cdot  \kappa_{\epsilon e}$ is the fibre integral along $\mathrm{pr}_1 : E_1 \times_{X} E_2 \to E_1$ of $\epsilon \times (\epsilon \cdot e) =  (\epsilon \times \epsilon) \cdot (1 \times e)$, so 
\begin{align*}
\lambda_{1,2}(\epsilon \cdot  \kappa_{\epsilon e}) &=
(\mathrm{pr}_1)_!(\lambda_{1,2}(\epsilon \times \epsilon) \cdot (1 \times e))\\
 &= (\mathrm{pr}_1)_!((\Delta_!(1) - 1 \times v - v \times 1 + p^*s^*e) \cdot (1 \times e))\\
&= e - \pi^*s^*e - \chi v + \chi \pi^*s^*e
\end{align*}
and hence
\begin{align*}
\lambda_{1,2}(\epsilon \times \tfrac{1}{\chi} \kappa_{\epsilon e}) &= \tfrac{1}{\chi} (e - \pi^*s^*e - \chi v + \chi \pi^*s^*e) \times 1 \\
&= \tfrac{1}{\chi}e \times 1 + \tfrac{\chi-1}{\chi}p^*s^* e - v \times 1
\end{align*}
and similarly
$$\lambda_{1,2}( \tfrac{1}{\chi} \kappa_{\epsilon e} \times \epsilon) = \tfrac{1}{\chi}1 \times e + \tfrac{\chi-1}{\chi}p^*s^* e - 1 \times v.$$

Combining these gives
\begin{align*}
\lambda_{1,2}(\bar{\epsilon} \times \bar{\epsilon}) &= \Delta_!(1) - 1 \times v - v \times 1 + p^*s^* e\\
&\quad\quad+ \tfrac{1}{\chi^2} p^*\kappa_{e^2} + \tfrac{\chi-2}{\chi}p^*s^*e\\
&\quad\quad- (\tfrac{1}{\chi}e \times 1 + \tfrac{\chi-1}{\chi}p^*s^* e - v \times 1)\\
&\quad\quad-(\tfrac{1}{\chi}1 \times e + \tfrac{\chi-1}{\chi}p^*s^* e - 1 \times v)\\
&= \Delta_!(1)  + \tfrac{1}{\chi^2} p^*\kappa_{e^2} - \tfrac{1}{\chi}(e \times 1 + 1 \times e)
\end{align*}
as required.
\end{proof}

If in addition we have a lift $\ell : E \to B$ of the fibrewise Gauss map along some fibration $\theta : B \to B\SO(2n)$ then for any $c \in H^*(B;\bQ)$ we can define modified twisted Miller--Morita--Mumford classes by the formula 
$$\kappa_{\bar{\epsilon}^a c} := \pi_!(\bar{\epsilon}^a \cdot \ell^*c) \in H^{n  (a-2) + |c|}(X ; \cH^{\otimes a}).$$
Under the action of a permutation $\sigma \in \mathfrak{S}_a$ of the tensor factors these classes transform as $\mathrm{sign}(\sigma)^n$, as $\bar{\epsilon}$ has degree $n$. Thus for any finite set $S$ there is a well-defined element
\begin{equation}\label{eq:Unorderedkappa}
\kappa_{\bar{\epsilon}^S c} := \pi_!(\bar{\epsilon}^a \cdot \ell^*c) \in H^{n  (a-2) + |c|}(X ; \cH^{\otimes S}) \otimes (\det \bQ^S)^{\otimes n}.
\end{equation}
To keep track of signs, for an ordered set $S = \{s_1 < s_2 < \cdots < s_a\}$  we will often write $\kappa_{\bar{\epsilon}^{s_1, \ldots, s_a} c} \in H^{n  (a-2) + |c|}(X ; \cH^{\otimes S})$ for the corresponding element, understanding that if $\sigma$ is a reordering of $S$ then $\kappa_{\bar{\epsilon}^{\sigma(s_1), \ldots, \sigma(s_a)} c} = \mathrm{sign}(\sigma)^n\kappa_{\bar{\epsilon}^{s_1, \ldots, s_a} c}$.

Using Lemma \ref{lem:ModifiedContractionFormula} we immediately see that these characteristic classes satisfy the following analogue of the contraction formula from \cite[Proposition 3.10]{KR-WTorelli}.

\begin{proposition}[Modified contraction formula]\label{prop:MCF}
In $H^*(X ; \cH^{\otimes -})$ we have the identities
\begin{align*}
\lambda_{1,2}(\pi_!(\bar{\epsilon}^{1,2,\ldots,a} \cdot \ell^*c)) &= (\tfrac{\chi-2}{\chi})\pi_!(\bar{\epsilon}^{3,4,\ldots,a} \cdot \ell^*(e \cdot c)) + \tfrac{1}{\chi^2}\kappa_{e^2} \cdot \pi_!(\bar{\epsilon}^{3,4,\ldots,a} \cdot \ell^*c)
\end{align*}
and
\begin{align*}
\lambda_{a,a+1}(\pi_!(\bar{\epsilon}^{1,2,\ldots, a} \cdot \ell^*c) &\cdot \pi_!(\bar{\epsilon}^{a+1, \ldots, a+b} \cdot \ell^*c')) = \pi_!(\bar{\epsilon}^{1,\ldots, a-1, a+2, \ldots, a+b} \cdot \ell^*(c \cdot c'))\\
& \quad+ \tfrac{1}{\chi^2} \kappa_{e^2} \cdot \pi_!(\bar{\epsilon}^{1,\ldots,a-1} \cdot \ell^*c) \cdot \pi_!(\bar{\epsilon}^{a+2, \ldots, a+b} \cdot \ell^*c')\\
& \quad- \tfrac{1}{\chi} \pi_!(\bar{\epsilon}^{1,\ldots, a-1} \cdot \ell^*(e \cdot c)) \cdot \pi_!(\bar{\epsilon}^{a+2, \ldots, a+b} \cdot \ell^*c')\\
& \quad- \tfrac{1}{\chi} \pi_!(\bar{\epsilon}^{1,\ldots, a-1} \cdot \ell^*c) \cdot \pi_!(\bar{\epsilon}^{a+2, \ldots, a+b} \cdot \ell^*(e \cdot c')).
\end{align*}%\qed
\end{proposition}

Similarly, from Lemma \ref{lem:Compare} we immediately obtain the following:

\begin{proposition}\label{prop:Pullback}
If the bundle $\pi : E \to X$ has a section, so that the class $\epsilon$ and hence $\kappa_{\epsilon^S c}$ is defined, then
\begin{equation*}
\kappa_{\bar{\epsilon}^{S} c}  =  \sum_{I \subseteq S}\kappa_{\epsilon^I c}(- \tfrac{1}{\chi} \kappa_{\epsilon e})^{S \setminus I} \in H^*(X ; \cH^{\otimes S}) \otimes (\det \bQ^S)^{\otimes n}.
\end{equation*}
\end{proposition}

Let us give an example of using the modified contraction formula to evaluate an expression.

\begin{example}\label{ex:theta}
Consider the class $\lambda_{1,5}\lambda_{2,6} \lambda_{3,4}(\kappa_{\bar{\epsilon}^{1,2,3}} \cdot \kappa_{\bar{\epsilon}^{4,5,6}})$. Then
\begin{align*}
\lambda_{1,5}\lambda_{2,6} \lambda_{3,4}(\kappa_{\bar{\epsilon}^{1,2,3}} \cdot \kappa_{\bar{\epsilon}^{4,5,6}}) &= \lambda_{1,5}\lambda_{2,6}\Bigl(\kappa_{\bar{\epsilon}^{1,2,5,6}} + \tfrac{1}{\chi^2} \kappa_{e^2} \kappa_{\bar{\epsilon}^{1,2}}\kappa_{\bar{\epsilon}^{5,6}} \\
& \quad\quad\quad\quad\quad\quad\quad\quad\quad - \tfrac{1}{\chi}(\kappa_{\bar{\epsilon}^{1,2}e}\kappa_{\bar{\epsilon}^{5,6}} + \kappa_{\bar{\epsilon}^{1,2}}\kappa_{\bar{\epsilon}^{5,6}e})\Bigr).
\end{align*}

The first term is
\begin{align*}
\lambda_{1,5}\lambda_{2,6}(\kappa_{\bar{\epsilon}^{1,2,5,6}}) &= (-1)^n\lambda_{1,5}\lambda_{2,6}(\kappa_{\bar{\epsilon}^{2,6,1,5}})\\
&= (-1)^n\lambda_{1,5}(\tfrac{\chi-2}{\chi} \kappa_{\bar{\epsilon}^{1,5}e} + \tfrac{1}{\chi^2} \kappa_{e^2} \kappa_{\bar{\epsilon}^{1,5}})\\
&= (-1)^n \tfrac{\chi-2}{\chi}(\tfrac{\chi-2}{\chi} \kappa_{e^2} + \tfrac{1}{\chi^2}\kappa_{e^2} \chi ) + (-1)^n \tfrac{1}{\chi^2} \kappa_{e^2} (\tfrac{\chi-2}{\chi} \chi )\\
&=(-1)^n(\tfrac{(\chi-2)^2}{\chi^2} + 2\tfrac{\chi-2}{\chi^2}) \kappa_{e^2}.
\end{align*}
The second term is
\begin{align*}
\tfrac{1}{\chi^2} \kappa_{e^2} \lambda_{1,5}\lambda_{2,6}(\kappa_{\bar{\epsilon}^{1,2}}\kappa_{\bar{\epsilon}^{5,6}}) &= (-1)^n\tfrac{1}{\chi^2} \kappa_{e^2} \lambda_{1,5}\lambda_{2,6}(\kappa_{\bar{\epsilon}^{1,2}}\kappa_{\bar{\epsilon}^{6,5}})\\
&= (-1)^n\tfrac{1}{\chi^2} \kappa_{e^2} \lambda_{1,5}(\kappa_{\bar{\epsilon}^{1,5}})\\
&= (-1)^n\tfrac{\chi-2}{\chi^2} \kappa_{e^2}.
\end{align*}
The third term is
\begin{align*}
-\tfrac{1}{\chi} \lambda_{1,5}\lambda_{2,6}(\kappa_{\bar{\epsilon}^{1,2}e}\kappa_{\bar{\epsilon}^{5,6}}) &= (-1)^{n+1}\tfrac{1}{\chi} \lambda_{1,5}\lambda_{2,6}(\kappa_{\bar{\epsilon}^{1,2}e}\kappa_{\bar{\epsilon}^{6,5}})\\
&= (-1)^{n+1}\tfrac{1}{\chi} \lambda_{1,5}(\kappa_{\bar{\epsilon}^{1,5}e} - \tfrac{1}{\chi} \kappa_{\bar{\epsilon}^{1}e}\kappa_{\bar{\epsilon}^{5}e})\\
&= (-1)^{n+1}\tfrac{1}{\chi} \Bigl((\tfrac{\chi-2}{\chi} \kappa_{e^2} + \tfrac{1}{\chi^2}\kappa_{e^2} \chi) \\
& \quad\quad\quad\quad\quad\quad\quad\quad\quad - \tfrac{1}{\chi}(\kappa_{e^2} + \tfrac{1}{\chi^2} \kappa_{e^2} \chi^2 - \tfrac{1}{\chi}(2 \chi \kappa_{e^2}))\Bigr)\\
&= (-1)^{n+1}(\tfrac{\chi-2}{\chi^2} + \tfrac{1}{\chi^2}-\tfrac{1}{\chi^2} - \tfrac{1}{\chi^2}+\tfrac{2}{\chi^2})\kappa_{e^2}\\
&= (-1)^{n+1}\tfrac{\chi-1}{\chi^2}\kappa_{e^2}
\end{align*}
and the fourth term is the same as the third by the evident symmetry.

In total we have
\begin{align*}
\lambda_{1,5}\lambda_{2,6} \lambda_{3,4}(\kappa_{\bar{\epsilon}^{1,2,3}} \cdot \kappa_{\bar{\epsilon}^{4,5,6}}) &= (-1)^{n}(\tfrac{(\chi-2)^2}{\chi^2} + 2\tfrac{\chi-2}{\chi^2} +\tfrac{\chi-2}{\chi^2} - 2\tfrac{\chi-1}{\chi^2}) \kappa_{e^2}\\
&= (-1)^{n}\tfrac{\chi-3}{\chi} \kappa_{e^2}.
\end{align*}
\qed
\end{example}

\subsection{Graphical interpretation}\label{sec:Graphical}

In \cite[Section 5]{KR-WTorelli} it was found to be very convenient to adopt a graphical formalism where $\kappa_{\epsilon^a c}$ corresponds to a vertex with $a$ half-edges incident to it and a formal label $c$, a product of $\kappa_{\epsilon^a c}$'s corresponds to a disjoint union of such vertices, and applying the contraction $\lambda_{i,j}$ corresponds to pairing up the half-edges labelled $i$ and $j$.

It will be convenient to adopt a similar formalism here. Let $S$ be a finite set, and $\mathcal{V}$ be a graded $\bQ$-algebra with a distinguished element $e \in \mathcal{V}_{2n}$. Slightly modifying\footnote{The difference is that we allow labelled vertices whose contribution to the degree is 0.} the definition from \cite[Proof of Theorem 5.1]{KR-WTorelli}, a \emph{marked oriented graph} with legs $S$ and labelled by $\cV$ consists of the following data:
\begin{enumerate}[\indent (i)]
\item a totally ordered finite set $\vec{V}$ (of vertices), a totally ordered finite set $\vec{H}$ (of half-edges), and a monotone function $a \colon \vec{H} \to \vec{V}$ (encoding that a half-edge $h$ is incident to the vertex $a(h)$),
\item an ordered matching $m = \{(a_i, b_i)\}_{i \in I}$ of the set $H \sqcup S$ (encoding the oriented edges of the graph),
\item a function $c \colon V \to \cV$ with homogeneous values, such that $|c(v)| + n(|a^{-1}(v)|-2) \geq 0$.
\end{enumerate}

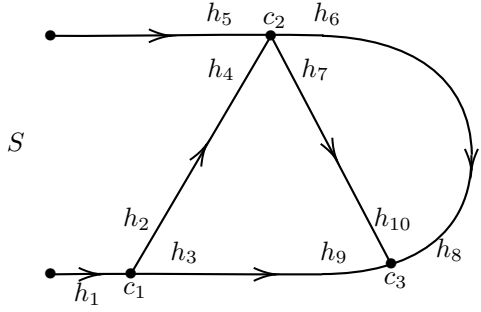
\begin{figure}[h]

\begin{tikzpicture}[x=0.75pt,y=0.75pt,yscale=-1,xscale=1]

%Curve Lines [id:da10107348461231513] 
\draw [thick]   (90,50) .. controls (134.99,50.3) and (182.06,48.81) .. (226.26,49.8)(225.51,50) .. controls (268.95,51.19) and (299.47,70.22) .. (300.22,109.83)(300,110) .. controls (300.75,149.62) and (270.45,169.85) .. (225.51,169.85)(225.51,170) .. controls (180.56,170.15) and (135.45,169.34) .. (90,170) ;
%Straight Lines [id:da9832338573338575] 
\draw [thick]   (130,170) -- (200,50) ;
\draw [thick, shift={(167.85,104.64)}, rotate = 120.14] [color={rgb, 255:red, 0; green, 0; blue, 0 }  ][line width=0.75]    (10.93,-3.29) .. controls (6.95,-1.4) and (3.31,-0.3) .. (0,0) .. controls (3.31,0.3) and (6.95,1.4) .. (10.93,3.29)   ;
%Straight Lines [id:da823753850443907] 
\draw  [thick]  (200,50) -- (260,165) ;
\draw [thick, shift={(232.81,112.47)}, rotate = 242.12] [color={rgb, 255:red, 0; green, 0; blue, 0 }  ][line width=0.75]    (10.93,-3.29) .. controls (6.95,-1.4) and (3.31,-0.3) .. (0,0) .. controls (3.31,0.3) and (6.95,1.4) .. (10.93,3.29)   ;
%Straight Lines [id:da631966871382807] 
%\draw  [thick]  (190,170) -- (202.67,170) ;
\draw [thick, shift={(202.33,170)}, rotate = 180] [color={rgb, 255:red, 0; green, 0; blue, 0 }  ][line width=0.75]    (10.93,-3.29) .. controls (6.95,-1.4) and (3.31,-0.3) .. (0,0) .. controls (3.31,0.3) and (6.95,1.4) .. (10.93,3.29)   ;
%Straight Lines [id:da18605736867905043] 
%\draw  [thick]  (300,110) -- (300,113.67) ;
\draw [thick, shift={(300,117.83)}, rotate = 270] [color={rgb, 255:red, 0; green, 0; blue, 0 }  ][line width=0.75]    (10.93,-3.29) .. controls (6.95,-1.4) and (3.31,-0.3) .. (0,0) .. controls (3.31,0.3) and (6.95,1.4) .. (10.93,3.29)   ;

%Straight Lines [id:da631966871382807] 
%\draw  [thick]  (95,170) -- (202.67,170) ;
\draw [thick, shift={(115,170)}, rotate = 180] [color={rgb, 255:red, 0; green, 0; blue, 0 }  ][line width=0.75]    (10.93,-3.29) .. controls (6.95,-1.4) and (3.31,-0.3) .. (0,0) .. controls (3.31,0.3) and (6.95,1.4) .. (10.93,3.29)   ;

\draw [thick, shift={(150,50)}, rotate = 180] [color={rgb, 255:red, 0; green, 0; blue, 0 }  ][line width=0.75]    (10.93,-3.29) .. controls (6.95,-1.4) and (3.31,-0.3) .. (0,0) .. controls (3.31,0.3) and (6.95,1.4) .. (10.93,3.29)   ;

% Text Node
\draw (125,173) node [anchor=north west][inner sep=0.75pt]    {$c_{1}$};
% Text Node
\draw (194.67,37) node [anchor=north west][inner sep=0.75pt]    {$c_{2}$};
% Text Node
\draw (255,168) node [anchor=north west][inner sep=0.75pt]    {$c_{3}$};
% Text Node
\draw (67,97.4) node [anchor=north west][inner sep=0.75pt]    {$S$};
% Text Node
\draw (100.33,173.07) node [anchor=north west][inner sep=0.75pt]    {$h_{1}$};
% Text Node
\draw (125,135) node [anchor=north west][inner sep=0.75pt]    {$h_{2}$};
% Text Node
\draw (148.67,152) node [anchor=north west][inner sep=0.75pt]    {$h_{3}$};
% Text Node
\draw (167,60) node [anchor=north west][inner sep=0.75pt]    {$h_{4}$};
% Text Node
\draw (166.33,32) node [anchor=north west][inner sep=0.75pt]    {$h_{5}$};
% Text Node
\draw (220,32) node [anchor=north west][inner sep=0.75pt]    {$h_{6}$};
% Text Node
\draw (214,60) node [anchor=north west][inner sep=0.75pt]    {$h_{7}$};
% Text Node
\draw (281,150.73) node [anchor=north west][inner sep=0.75pt]    {$h_{8}$};
% Text Node
\draw (223.33,152) node [anchor=north west][inner sep=0.75pt]    {$h_{9}$};
% Text Node
\draw (250,135) node [anchor=north west][inner sep=0.75pt]    {$h_{10}$};

\node at (90,50)  {$\bullet$};
\node at (90,170)  {$\bullet$};

\node at (130,170)  {$\bullet$};
\node at (200,50)  {$\bullet$};
\node at (260,165)  {$\bullet$};

\end{tikzpicture}
\caption{An example of a marked oriented graph, with $\vec{V} = (v_1, v_2, v_3)$, $c_i = c(v_i) \in \cV$, and $\vec{H} = (h_1, h_2, \ldots, h_{10})$.}\label{fig:ExMarkedGph}
\end{figure}

\begin{convention*}
In later figures, to avoid clutter we adopt the following ordering conventions: vertices are numbered starting from 1 from left to right, half-edges around each vertex are ordered clockwise starting from the marked half-edge, and edges are oriented from the smaller half-edge to the larger one. These conventions are already used in Figure \ref{fig:ExMarkedGph}, when the ``marked half-edge'' is taken to be the smallest half-edge at a given vertex.
\end{convention*}

Marked oriented graphs $\Gamma = (\vec{V}, \vec{H}, a, m, c)$ and $\Gamma' = (\vec{V}', \vec{H}', a', m', c')$ with the same set of legs $S$ are \emph{isomorphic} if there are order-preserving bijections $\vec{V} \overset{\sim}\to \vec{V}'$ and $\vec{H} \overset{\sim}\to \vec{H}'$ which intertwine $a$ and $a'$, intertwine $c$ and $c'$, and send $m$ to $m'$. An \emph{oriented graph} is an isomorphism class $[\Gamma]$ of marked oriented graphs. We assign to a marked oriented graph $\Gamma = (\vec{V}, \vec{H}, a, m, c)$ the degree
\[\mathrm{deg}(\Gamma) \coloneqq \sum_{v \in V} \left(|c(v)| + n(|a^{-1}(v)|-2)\right) = n (|H|-2|V|) + \sum_{v \in V} |c(v)|.\]

Let $\pi : E \to X$ be an oriented $W_g$-bundle with a lift $\ell : E \to B$ of the map classifying the vertical tangent bundle along $\theta : B \to B\SO(2n)$, and let $\mathcal{V} := H^*(B;\bQ)$ and $e := \theta^* e \in \mathcal{V}_{2n}$. Then given a marked oriented graph $\Gamma = (\vec{V}, \vec{H}, a, m, c)$ with legs $S$ we form a class 
$$\bar{\kappa}(\Gamma) \in H^{\mathrm{deg}(\Gamma)}(X;\cH^{\otimes S})$$
by the following recipe. Firstly, we may form
\begin{equation}\label{eq:VxContrib}
\prod_{v \in V} \kappa_{\bar{\epsilon}^{a^{-1}(v)} c(v)} \in H^*(X; \cH^{\otimes H}),
\end{equation}
where we have used the ordering on $\vec{V}$ to order the product, and the ordering on $\vec{H}$ to trivialise the factor of $(\det \bQ^H)^{\otimes n} = (\bigotimes_{v \in V} \det \bQ^{a^{-1}(v)})^{\otimes n}$ that arises from \eqref{eq:Unorderedkappa}. Secondly, taking two copies $S_1$ and $S_2$ of the set $S$ and writing $s_i \in S_i$ for the element corresponding to $s \in S$ we can form
\begin{equation}\label{eq:LegContrib}
\prod_{s \in S} \kappa_{\bar{\epsilon}^{s_1, s_2}} \in H^*(X, \mathcal{H}^{\otimes (S_1 \sqcup S_2)}).
\end{equation}
As each $\kappa_{\bar{\epsilon}^{s_1, s_2}}$ has degree 0, the product does not depend on how the factors are ordered. Taking the product of \eqref{eq:VxContrib} and \eqref{eq:LegContrib} and then applying $\lambda_{x,y}$ for each ordered pair $(x,y)$ in the matching $m$ on $H \sqcup S = H \sqcup S_1$ gives the required class $\bar{\kappa}(\Gamma) \in H^{*}(X;\cH^{\otimes S_2}) = H^{*}(X;\cH^{\otimes S})$. 

\begin{example}
In this graphical interpretation we recognise the class evaluated in Example \ref{ex:theta} as that associated to the theta-graph $\begin{tikzpicture}
 [baseline=-.65ex]
 \node[circle,draw,fill,inner sep=1pt] (v) at (0,0) {};
 \node[circle,draw,fill,inner sep=1pt] (w) at (0.5,0) {};
\draw (v) edge[bend left] (w) edge[bend right] (w) edge (w);
\end{tikzpicture}$ with a certain ordering and orientation.
\end{example}

Clearly $\bar{\kappa}(\Gamma)$ only depends on the underlying oriented graph $[\Gamma]$. We now describe how it transforms when the orderings on $V$, $H$, and the pairs $m$ are changed, without changing the underlying labelled graph. If $\Gamma'=(\vec{V}', \vec{H}', a', m', c')$ is another marked oriented graph and there are bijections $f : H \to H'$ and $g : V \to V'$ intertwining $a$ and $a'$ and $c$ and $c'$ and such that under these bijections the matching $m'$ differs from $m$ by reversing the order of $k$ pairs, then
$$\bar{\kappa}(\Gamma') = (-1)^{nk} \sign(f) \sign(g) \cdot \bar{\kappa}(\Gamma)$$
for certain signs described on \cite[pp.\ 55-56]{KR-WTorelli}.

Graphs considered as representing $\bar{\kappa}$'s behave differently to those representing $\kappa$'s described in \cite[Section 5]{KR-WTorelli}. To distinguish them we will depict the graphs representing $\kappa$'s in \textcolor{Mahogany}{red}, as we did in that paper, and the graphs representing $\bar{\kappa}$'s in \textcolor{NavyBlue}{blue}. The contraction formula of \cite[Proposition 3.10]{KR-WTorelli} was interpreted in \cite[Section 5]{KR-WTorelli} as giving relations among \textcolor{Mahogany}{red} graphs which yield equivalent $\kappa$-classes. In the generality of a smooth oriented $W_g$-bundle $\pi : E \to X$ with section $s : X \to E$ these may be depicted as follows:

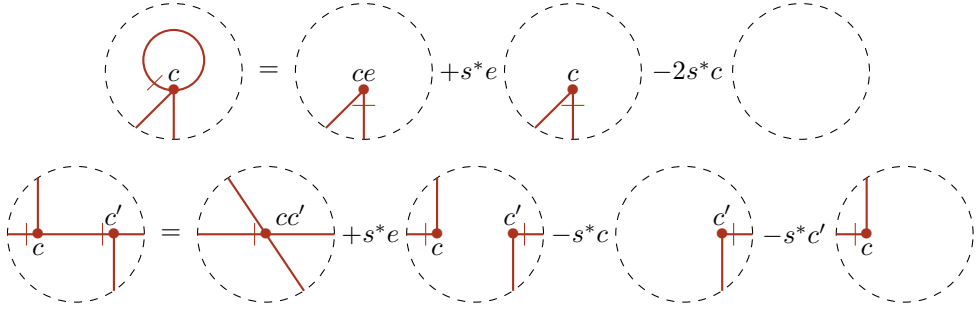
\begin{figure}[h]

\begin{tikzpicture}[arrowmark/.style 2 args={decoration={markings,mark=at position #1 with \arrow{#2}}}]
\begin{scope}[scale=0.5]

	\draw[dashed] (0,0) circle (1.8cm);
	\draw [thick,Mahogany]  (0,.3) circle (0.8cm);
	\node at (0,-0.5) [Mahogany] {$\bullet$};
	\draw [Mahogany] (-0.7,-0.5) -- (-0.3,-0.1);
	\draw [thick,Mahogany] (0,-1.8) -- (0,-0.5);
	\draw [thick,Mahogany] (-1,-1.5) -- (0,-0.5);
	\node at (0,-.5) [above] {$c$};

	\node at (2.5,0) {=};
	
	\begin{scope}[xshift=5cm]
	\draw[dashed] (0,0) circle (1.8cm);
	\node at (0,-0.5) [Mahogany] {$\bullet$};
		\draw [Mahogany] (-0.3,-0.9) -- (0.3,-0.9);
	\draw [thick,Mahogany] (0,-1.8) -- (0,-0.5);
	\draw [thick,Mahogany] (-1,-1.5) -- (0,-0.5);
	\node at (0,-.5) [above] {$ce$};
	
	\end{scope}
	
	\begin{scope}[xshift=10.5cm]
	\node at (-2.75,0) {$+ s^*e$};
	\draw[dashed] (0,0) circle (1.8cm);
	\node at (0,-0.5) [Mahogany] {$\bullet$};
	\draw [Mahogany] (-0.3,-0.9) -- (0.3,-0.9);
	\draw [thick,Mahogany] (0,-1.8) -- (0,-0.5);
	\draw [thick,Mahogany] (-1,-1.5) -- (0,-0.5);
	\node at (0,-.5) [above] {$c$};
	\end{scope}

	\begin{scope}[xshift=16.5cm]
	\node at (-3,0) {$- 2s^*c$};
	\draw[dashed] (0,0) circle (1.8cm);	
	\end{scope}

	\end{scope}
	\end{tikzpicture}

	\vspace{2ex}
	
	\begin{tikzpicture}
	\begin{scope}[scale=0.5]
	\draw[dashed] (0,0) circle (1.8cm);
	
	\node at (-1,0) [Mahogany] {$\bullet$};
	\draw [Mahogany] (-1.3,-0.3) -- (-1.3,0.3);
	\node at (-1,0) [below] {$c$};
	\node at (1,0) [Mahogany] {$\bullet$};
	\draw [Mahogany] (0.7,-0.3) -- (0.7,0.3);
	\node at (1,0) [above] {$c'$};
	\draw [thick,Mahogany] (-1,0) -- (-1.8,0);
	\draw [thick,Mahogany] (-1,0) -- (-1,1.5);
	\draw [thick,Mahogany] (-1,0) -- (1,0);
	\draw [thick,Mahogany] (1,0) -- (1.8,0);
	\draw [thick,Mahogany] (1,0) -- (1,-1.5);
	
	\node at (2.5,0) {=};
	
	\begin{scope}[xshift=5cm]
	\draw[dashed] (0,0) circle (1.8cm);
	\node at (0,0) [Mahogany] {$\bullet$};
	\draw [Mahogany] (-0.3,-0.3) -- (-0.3,0.3);
	\node at (0,0) [above right] {$c c'$};
	\draw [thick,Mahogany] (-1.8,0) -- (0,0);
	\draw [thick,Mahogany] (0,0) -- (-1,1.5);
	\draw [thick,Mahogany] (0,0) -- (1.8,0);
	\draw [thick,Mahogany] (0,0) -- (1,-1.5);

	\node at (2.75,0) {$+ s^*e$};

	\begin{scope}[xshift=5.5cm]
	\draw[dashed] (0,0) circle (1.8cm);
	\node at (-1,0) [Mahogany] {$\bullet$};
	\draw [Mahogany] (-1.3,-0.3) -- (-1.3,0.3);
	\node at (-1,0) [below] {$c$};
	\node at (1,0) [Mahogany] {$\bullet$};
	\draw [Mahogany] (1.3,-0.3) -- (1.3,0.3);
	\node at (1,0) [above] {$c'$};
	\draw [thick,Mahogany] (-1,0) -- (-1.8,0);
	\draw [thick,Mahogany] (-1,0) -- (-1,1.5);
	\draw [thick,Mahogany] (1,0) -- (1.8,0);
	\draw [thick,Mahogany] (1,0) -- (1,-1.5);
	\node at (2.75,0) {$- s^*c$};
	
	\begin{scope}[xshift=5.5cm]
	\draw[dashed] (0,0) circle (1.8cm);
	
	\node at (1,0) [Mahogany] {$\bullet$};
	\draw [Mahogany] (1.3,-0.3) -- (1.3,0.3);
	\node at (1,0) [above] {$c'$};
	\draw [thick,Mahogany] (1,0) -- (1.8,0);
	\draw [thick,Mahogany] (1,0) -- (1,-1.5);

	\node at (2.9,0) {$- s^*c'$};
	
	\begin{scope}[xshift=5.8cm]
	\draw[dashed] (0,0) circle (1.8cm);
	\node at (-1,0) [Mahogany] {$\bullet$};
	\draw [Mahogany] (-1.3,-0.3) -- (-1.3,0.3);
	\node at (-1,0) [below] {$c$};
	\draw [thick,Mahogany] (-1,0) -- (-1.8,0);
	\draw [thick,Mahogany] (-1,0) -- (-1,1.5);
	\end{scope}

	\end{scope}

	\end{scope}

	\end{scope}

	\end{scope}
	\end{tikzpicture}
\caption{The contraction formula, displayed graphically.}\label{fig:ContFormulaGraph}
\end{figure}
Here the negative terms only arise when they make sense, i.e.\ when the vertex has valence 2 in the first case, when the vertex labelled $c$ has valence 1 in the second case, and when the vertex labelled $c'$ has valence 1 in the third case. We have used the ordering convention described above, where the marked half-edge is indicated by a bar.

Similarly, the modified contraction formula of Proposition \ref{prop:MCF} can be interpreted as giving the following relations among \textcolor{NavyBlue}{blue} graphs which yield equivalent $\bar{\kappa}$-classes: 

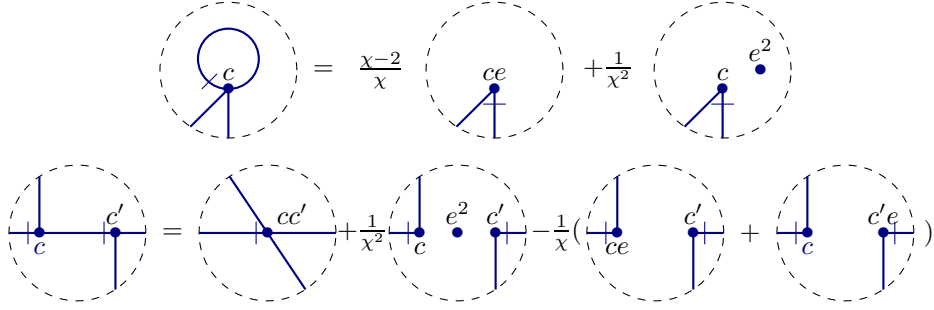
\begin{figure}[h]
	\begin{tikzpicture}[arrowmark/.style 2 args={decoration={markings,mark=at position #1 with \arrow{#2}}}]
\begin{scope}[scale=0.5]

	\draw[dashed] (0,0) circle (1.8cm);
	\draw [thick,NavyBlue]  (0,.3) circle (0.8cm);
	\node at (0,-0.5) [NavyBlue] {$\bullet$};
	\draw [NavyBlue] (-0.7,-0.5) -- (-0.3,-0.1);
	\draw [thick,NavyBlue] (0,-1.8) -- (0,-0.5);
	\draw [thick,NavyBlue] (-1,-1.5) -- (0,-0.5);
	\node at (0,-.5) [above] {$c$};

	\node at (2.5,0) {=};
	
	\begin{scope}[xshift=7cm]
	\node at (-3,0) {$\tfrac{\chi-2}{\chi}$};
	\draw[dashed] (0,0) circle (1.8cm);
	\node at (0,-0.5) [NavyBlue] {$\bullet$};
	\draw [NavyBlue] (-0.3,-0.9) -- (0.3,-0.9);
	\draw [thick,NavyBlue] (0,-1.8) -- (0,-0.5);
	\draw [thick,NavyBlue] (-1,-1.5) -- (0,-0.5);
	\node at (0,-.5) [above] {$ce$};
	\end{scope}
	
	\begin{scope}[xshift=13cm]
	\node at (-3,0) {$+\tfrac{1}{\chi^2}$};
	\draw[dashed] (0,0) circle (1.8cm);
	\node at (0,-0.5) [NavyBlue] {$\bullet$};
	\draw [NavyBlue] (-0.3,-0.9) -- (0.3,-0.9);
	\draw [thick,NavyBlue] (0,-1.8) -- (0,-0.5);
	\draw [thick,NavyBlue] (-1,-1.5) -- (0,-0.5);
	\node at (0,-.5) [above] {$c$};
	\node at (1,0) [NavyBlue] {$\bullet$};
	\node at (1,0) [above] {$e^2$};
	\end{scope}

	\end{scope}
	\end{tikzpicture}
	
	\vspace{2ex}
	
	\begin{tikzpicture}
	\begin{scope}[scale=0.5]
	\draw[dashed] (0,0) circle (1.8cm);
	
	\node at (-1,0) [NavyBlue] {$\bullet$};
	\draw [NavyBlue] (-1.3,-0.3) -- (-1.3,0.3);
	\node at (-1,0) [NavyBlue, below] {$c$};
	\node at (1,0) [NavyBlue] {$\bullet$};
	\draw [NavyBlue] (0.7,-0.3) -- (0.7,0.3);
	\node at (1,0) [above] {$c'$};
	\draw [thick,NavyBlue] (-1,0) -- (-1.8,0);
	\draw [thick,NavyBlue] (-1,0) -- (-1,1.5);
	\draw [thick,NavyBlue] (-1,0) -- (1,0);
	\draw [thick,NavyBlue] (1,0) -- (1.8,0);
	\draw [thick,NavyBlue] (1,0) -- (1,-1.5);
	
	\node at (2.5,0) {=};
	
	\begin{scope}[xshift=5cm]
	\draw[dashed] (0,0) circle (1.8cm);
	\node at (0,0) [NavyBlue] {$\bullet$};
	\draw [NavyBlue] (-0.3,-0.3) -- (-0.3,0.3);
	\node at (0,0) [above right] {$c c'$};
	\draw [thick,NavyBlue] (-1.8,0) -- (0,0);
	\draw [thick,NavyBlue] (0,0) -- (-1,1.5);
	\draw [thick,NavyBlue] (0,0) -- (1.8,0);
	\draw [thick,NavyBlue] (0,0) -- (1,-1.5);
	
	\node at (2.5,0) {$+ \tfrac{1}{\chi^2}$};

	\begin{scope}[xshift=5cm]
	\draw[dashed] (0,0) circle (1.8cm);
	
	\node at (-1,0) [NavyBlue] {$\bullet$};
	\draw [NavyBlue] (-1.3,-0.3) -- (-1.3,0.3);
	\node at (-1,0) [below] {$c$};
	\node at (1,0) [NavyBlue] {$\bullet$};
	\draw [NavyBlue] (1.3,-0.3) -- (1.3,0.3);
	\node at (1,0) [above] {$c'$};
	\draw [thick,NavyBlue] (-1,0) -- (-1.8,0);
	\draw [thick,NavyBlue] (-1,0) -- (-1,1.5);
	\draw [thick,NavyBlue] (1,0) -- (1.8,0);
	\draw [thick,NavyBlue] (1,0) -- (1,-1.5);
	\node at (0,0) [NavyBlue] {$\bullet$};
	\node at (0,0) [above] {$e^2$};
			
	\node at (2.6,0) {$- \tfrac{1}{\chi} ($};
	
	\begin{scope}[xshift=5.2cm]
	\draw[dashed] (0,0) circle (1.8cm);
	\node at (-1,0) [NavyBlue] {$\bullet$};
	\draw [NavyBlue] (-1.3,-0.3) -- (-1.3,0.3);
	\node at (-1,0) [below] {$ce$};
	\node at (1,0) [NavyBlue] {$\bullet$};
	\draw [NavyBlue] (1.3,-0.3) -- (1.3,0.3);
	\node at (1,0) [above] {$c'$};
	\draw [thick,NavyBlue] (-1,0) -- (-1.8,0);
	\draw [thick,NavyBlue] (-1,0) -- (-1,1.5);
	\draw [thick,NavyBlue] (1,0) -- (1.8,0);
	\draw [thick,NavyBlue] (1,0) -- (1,-1.5);
	
	\node at (2.5,0) {$+$};
	
	\begin{scope}[xshift=5cm]
	\draw[dashed] (0,0) circle (1.8cm);
	\node at (-1,0) [NavyBlue] {$\bullet$};
	\draw [NavyBlue] (-1.3,-0.3) -- (-1.3,0.3);
	\node at (-1,0) [NavyBlue, below] {$c$};
	\node at (1,0) [NavyBlue] {$\bullet$};
	\draw [NavyBlue] (1.3,-0.3) -- (1.3,0.3);
	\node at (1,0) [above] {$c'e$};
	\draw [thick,NavyBlue] (-1,0) -- (-1.8,0);
	\draw [thick,NavyBlue] (-1,0) -- (-1,1.5);
	\draw [thick,NavyBlue] (1,0) -- (1.8,0);
	\draw [thick,NavyBlue] (1,0) -- (1,-1.5);
	
	\node at (2.2,0) {$)$};

	\end{scope}

	\end{scope}

	\end{scope}

	\end{scope}

	\end{scope}
	\end{tikzpicture}
\caption{The modified contraction formula, displayed graphically.}\label{fig:ModContFormulaGraph}
\end{figure}

\section{Twisted cohomology of diffeomorphism groups}\label{sec:TwistedCoh}

The main goal of this section is to describe the twisted cohomology groups
$$H^*(B\Diff^{+}(W_g) ; \mathcal{H}^{\otimes S}) \text{ and } H^*(B\Diff^{+}(W_g, *) ; \mathcal{H}^{\otimes S})$$
in a stable range of degrees, of the classifying space $B\Diff^{+}(W_g)$ of the group of orientation-preserving diffeomorphisms of $W_g$ (which classifies oriented $W_g$-bundles), and the classifying space $B\Diff^{+}(W_g, *)$ of the group of orientation-preserving diffeomorphisms of $W_g$ which fix a point $* \in W_g$ (which classifies oriented $W_g$-bundles with section). In \cite[Theorem 3.15]{KR-WTorelli} the analogous calculation was given for the classifying space $B\Diff(W_g, D^{2n})$ of the group of  diffeomorphisms of $W_g$ which fix a disc $D^{2n} \subset W_g$.

In order to do this we will also discuss the manifolds $W_g$ equipped with $\theta$-structures for the tangential structure $\theta : B\SO(2n)\langle n \rangle \to B\OO(2n)$, i.e.\ the $n$-connected cover of $B\OO(2n)$. In this case we will consider the homotopy quotients
\begin{align*}
B\Diff^\theta(W_g) &:= \mathrm{Bun}(TW_g, \theta^*\gamma_{2n}) \hcoker \Diff(W_g)\\
B\Diff^\theta(W_g, *) &:= \mathrm{Bun}(TW_g, \theta^*\gamma_{2n}) \hcoker \Diff(W_g, *)
\end{align*}
where $\mathrm{Bun}(TW_g, \theta^*\gamma_{2n})$ denotes the space of vector bundle maps $TW_g \to \theta^*\gamma_{2n}$ from the tangent bundle of $W_g$ to the bundle classified by $\theta$. The group $\Diff(W_g)$ acts on the space of bundle maps by precomposing with the derivative.

There is a factorisation $\theta : B\SO(2n)\langle n \rangle \overset{\theta^{\mathrm{or}}}\to B\SO(2n) \overset{\sigma}\to B\OO(2n)$, and by obstruction theory one sees that the space $\mathrm{Bun}(TW_g, \sigma^*\gamma_{2n})$ has two contractible path components corresponding to the two orientations of $W_g$. In particular there are equivalences
\begin{align*}
\mathrm{Bun}(TW_g, \sigma^*\gamma_{2n}) \hcoker \Diff(W_g) &\simeq B\Diff^{+}(W_g)\\
\mathrm{Bun}(TW_g, \sigma^*\gamma_{2n}) \hcoker \Diff(W_g, *) &\simeq B\Diff^{+}(W_g, *)
\end{align*}
and so $\theta^\mathrm{or}$ induces maps
\begin{align*}
B\Diff^\theta(W_g) \lra B\Diff^{+}(W_g) \quad\text{ and }\quad B\Diff^\theta(W_g, *) \lra B\Diff^{+}(W_g, *).
\end{align*}
It is shown in \cite[Section 5.2]{grwsurvey} that these are principal $\SO[0,n-1]$-fibrations. In particular the spaces $B\Diff^\theta(W_g)$ and $B\Diff^\theta(W_g, *)$ are path-connected.

\subsection{Spaces of graphs}\label{sec:SpacesOfGraphs}

Our description of the twisted cohomology groups of $B\Diff^{+}(W_g)$, $B\Diff^{+}(W_g, *)$, $B\Diff(W_g, D^{2n}) $, $B\Diff^{\theta}(W_g)$ and $B\Diff^{\theta}(W_g, *)$ in a stable range will be---via the graphical interpretation given in Section \ref{sec:Graphical}---in terms of graded vector spaces of labelled graphs, modulo certain relations. (Readers of \cite{KR-WTorelli} may have been expecting vector spaces of labelled partitions instead: here we have found spaces of graphs more convenient for formulating results, cf.\ Remark \ref{rem:GraphsAndPartitions}, though spaces of labelled partitions will still play a role in the proofs.) To describe these spaces of graphs we will use the graded $\bQ$-algebras
\begin{align*}
\cV &\coloneqq H^*(B\SO(2n) \langle n \rangle;\bQ) = \bQ[p_{\lceil \tfrac{n+1}{4}\rceil}, \ldots, p_{n-1}, e]\\
\cW &\coloneqq H^*(B\SO(2n);\bQ) = \bQ[p_1, \ldots, p_{n-1}, e]
\end{align*}
with distinguished elements $e$ of degree $2n$ given by the Euler class. In order to work in a way which is agnostic about the genus $g$ of the manifold $W_g$ under consideration, we will work over the ring $\bQ[\chi^{\pm 1}]$ instead of $\bQ$, where $\chi$ is an invertible formal parameter which will---later---be set to the Euler characteristic $2 + (-1)^n 2g$ of $W_g$.

\begin{definition}\label{defn:GraphSpaces} \mbox{}
\begin{enumerate}[(i)]
\item Let 
$$\mathcal{G}\mathrm{raph}_1(S) := \bQ[\chi^{\pm 1}]\{\Gamma \text{ oriented graph with legs $S$, labelled in $\cV$}\}/\sim$$
where $\sim$ 
\begin{enumerate}[(a)]
\item imposes the sign rule for changing orderings of vertices and half-edges and for reversing orientations of edges;

\item imposes linearity in the labels, and sets a graph containing an $a$-valent vertex labelled by $c$ with $|c| + n(a-1) < 0$ to zero;

\item sets the 0-valent vertex labelled by $e \in \cV_{2n}$ equal to $\chi$, and if $2n \equiv 0 \mod 4$ sets the 0-valent vertex labelled by $p_{n/2} \in \cV_{2n}$ equal to 0;

\item imposes the contraction relations\footnote{These are the relations from Figure \ref{fig:ContFormulaGraph} when $s^*$ kills all positive-degree classes.}

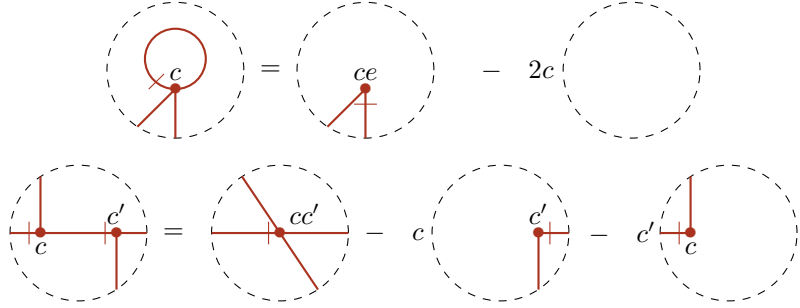
\begin{figure}[!h]

\begin{tikzpicture}[arrowmark/.style 2 args={decoration={markings,mark=at position #1 with \arrow{#2}}}]
\begin{scope}[scale=0.5]

	\draw[dashed] (0,0) circle (1.8cm);
	\draw [thick,Mahogany]  (0,.3) circle (0.8cm);
	\node at (0,-0.5) [Mahogany] {$\bullet$};
	\draw [Mahogany] (-0.7,-0.5) -- (-0.3,-0.1);
	\draw [thick,Mahogany] (0,-1.8) -- (0,-0.5);
	\draw [thick,Mahogany] (-1,-1.5) -- (0,-0.5);
	\node at (0,-.5) [above] {$c$};
	
	\node at (2.5,0) {=};
	
	\begin{scope}[xshift=5cm]
	\draw[dashed] (0,0) circle (1.8cm);
	\node at (0,-0.5) [Mahogany] {$\bullet$};
	\draw [Mahogany] (-0.3,-0.9) -- (0.3,-0.9);
	\draw [thick,Mahogany] (0,-1.8) -- (0,-0.5);
	\draw [thick,Mahogany] (-1,-1.5) -- (0,-0.5);
	\node at (0,-.5) [above] {$ce$};
	\end{scope}
	
	\begin{scope}[xshift=12cm]
	\node at (-3,0) {$- \quad 2c$};
	\draw[dashed] (0,0) circle (1.8cm);	
	\end{scope}

	\end{scope}
	\end{tikzpicture}
	
	\vspace{2ex}
	
	\begin{tikzpicture}
	\begin{scope}[scale=0.5]
	\draw[dashed] (0,0) circle (1.8cm);
	\node at (-1,0) [Mahogany] {$\bullet$};
	\draw [Mahogany] (-1.3,-0.3) -- (-1.3,0.3);
	\node at (-1,0) [below] {$c$};
	\node at (1,0) [Mahogany] {$\bullet$};
	\draw [Mahogany] (0.7,-0.3) -- (0.7,0.3);
	\node at (1,0) [above] {$c'$};
	\draw [thick,Mahogany] (-1,0) -- (-1.8,0);
	\draw [thick,Mahogany] (-1,0) -- (-1,1.5);
	\draw [thick,Mahogany] (-1,0) -- (1,0);
	\draw [thick,Mahogany] (1,0) -- (1.8,0);
	\draw [thick,Mahogany] (1,0) -- (1,-1.5);
	
	\node at (2.5,0) {=};
	
	\begin{scope}[xshift=5.3cm]
	\draw[dashed] (0,0) circle (1.8cm);
	\node at (0,0) [Mahogany] {$\bullet$};
	\draw [Mahogany] (-0.3,-0.3) -- (-0.3,0.3);
	\node at (0,0) [above right] {$c c'$};
	\draw [thick,Mahogany] (-1.8,0) -- (0,0);
	\draw [thick,Mahogany] (0,0) -- (-1,1.5);
	\draw [thick,Mahogany] (0,0) -- (1.8,0);
	\draw [thick,Mahogany] (0,0) -- (1,-1.5);
	
	\node at (3,0) {$-\quad c$};
	
	\begin{scope}[xshift=5.8cm]
	\draw[dashed] (0,0) circle (1.8cm);
	\node at (1,0) [Mahogany] {$\bullet$};
	\draw [Mahogany] (1.3,-0.3) -- (1.3,0.3);
	\node at (1,0) [above] {$c'$};
	\draw [thick,Mahogany] (1,0) -- (1.8,0);
	\draw [thick,Mahogany] (1,0) -- (1,-1.5);
	
	\node at (3.2,0) {$-\quad c'$};
	
	\begin{scope}[xshift=6cm]
	\draw[dashed] (0,0) circle (1.8cm);
	\node at (-1,0) [Mahogany] {$\bullet$};
	\draw [Mahogany] (-1.3,-0.3) -- (-1.3,0.3);
	\node at (-1,0) [below] {$c$};
	\draw [thick,Mahogany] (-1,0) -- (-1.8,0);
	\draw [thick,Mahogany] (-1,0) -- (-1,1.5);
	\end{scope}

	\end{scope}
		
	\end{scope}

	\end{scope}
	\end{tikzpicture}
\end{figure}
where the negative terms only arises when they makes sense,  i.e.\ in the first case when the vertex has valence 2 and its label $c$ is a scalar multiple of $1 \in \cV_0$, in the second case when $c$ is a scalar multiple of $1 \in \mathcal{V}_0$ and has valence 1, and similarly in the third case.
\end{enumerate}

\item Let 
$$\mathcal{G}\mathrm{raph}_*^\theta(S) := \bQ[\chi^{\pm 1}]\{\Gamma \text{ oriented graph with legs $S$, labelled in $\cV$}\}\otimes \cV/\sim$$
where $\sim$ imposes (a)--(c) as well as
\begin{enumerate}[(a$^\prime$)]
\setcounter{enumii}{3}

\item imposes the contraction relations of Figure \ref{fig:ContFormulaGraph}.
\end{enumerate}

\item Let 
$$\mathcal{G}\mathrm{raph}_*(S) := \bQ[\chi^{\pm 1}]\{\Gamma \text{ oriented graph with legs $S$, labelled in $\cW$}\}\otimes \cW /\sim$$
where $\sim$ imposes (a) and (b), as well as 

\begin{enumerate}[(a$^{\prime\prime}$)]
\setcounter{enumii}{2}

\item  sets the 0-valent vertex labelled by $e \in \cW_{2n}$ equal to $\chi$, sets the 0-valent vertex labelled by any degree $2n$ monomial in Pontrjagin classes equal to 0, and for any $1 \leq i  \leq \lfloor n/4 \rfloor$ sets

\begin{figure}[h]
	\begin{tikzpicture}[arrowmark/.style 2 args={decoration={markings,mark=at position #1 with \arrow{#2}}}]
	
	\begin{scope}[scale=0.5]
	\draw[dashed] (0,0) circle (1.8cm);
	\node at (0,-0.5) [Mahogany] {$\bullet$};
	\draw [Mahogany] (-0.3,-0.9) -- (0.3,-0.9);
	\draw [thick,Mahogany] (0,-1.8) -- (0,-0.5);
	\draw [thick,Mahogany] (-1,-1.5) -- (0,-0.5);
	\node at (0,-.5) [above] {$c p_i$};
	
	\begin{scope}[xshift=6cm]
	\node at (-3,0) {$= \tfrac{1}{\chi}$};
		
	\draw[dashed] (0,0) circle (1.8cm);
	\node at (0,-0.5) [Mahogany] {$\bullet$};
	\draw [Mahogany] (-0.3,-0.9) -- (0.3,-0.9);
	\draw [thick,Mahogany] (0,-1.8) -- (0,-0.5);
	\draw [thick,Mahogany] (-1,-1.5) -- (0,-0.5);
	\node at (0,-.5) [above] {$c$};
	
	\node at (2.5,0) {$\otimes p_i$};
	
	\end{scope}

	\end{scope}
	\end{tikzpicture}
\end{figure}
\end{enumerate}
and (d$^\prime$).

\item Let 
$$\mathcal{G}\mathrm{raph}^\theta(S) := \bQ[\chi^{\pm 1}]\{\Gamma \text{ oriented graph with legs $S$, labelled in $\cV$}\}/\sim$$
where $\sim$ imposes (a) and (b), as well as
\begin{enumerate}[(a$^{\prime\prime\prime}$)]
\setcounter{enumii}{2}

\item  sets the 0-valent vertex labelled by $e \in \cV_{2n}$ equal to $\chi$, if $2n \equiv 0 \mod 4$ sets the 0-valent vertex labelled by $p_{n/2} \in \cV_{2n}$ equal to 0, and sets the 1-valent vertex labelled by $e \in \cV_{2n}$ equal to 0,

\item imposes the contraction relations of Figure \ref{fig:ModContFormulaGraph}.
\end{enumerate}

\item Let 
$$\mathcal{G}\mathrm{raph}(S) := \bQ[\chi^{\pm 1}]\{\Gamma \text{ oriented graph with legs $S$, labelled in $\cW$}\}/\sim$$
where $\sim$ imposes (a), (b), as well as
\begin{enumerate}[(a$^{\prime\prime\prime\prime}$)]
\setcounter{enumii}{2}

\item  sets the 0-valent vertex labelled by $e \in \cW_{2n}$ equal to $\chi$, sets the 0-valent vertex labelled by any degree $2n$ monomial in Pontrjagin classes equal to 0, sets the 1-valent vertex labelled by $e \in \cW_{2n}$ equal to 0, and for any $1 \leq i  \leq \lfloor n/4 \rfloor$ sets

\begin{figure}[h]
	\begin{tikzpicture}[arrowmark/.style 2 args={decoration={markings,mark=at position #1 with \arrow{#2}}}]
	
	\begin{scope}[scale=0.5]
	\draw[dashed] (0,0) circle (1.8cm);
	\node at (0,-0.5) [NavyBlue] {$\bullet$};
	\draw [NavyBlue] (-0.3,-0.9) -- (0.3,-0.9);
	\draw [thick,NavyBlue] (0,-1.8) -- (0,-0.5);
	\draw [thick,NavyBlue] (-1,-1.5) -- (0,-0.5);
	\node at (0,-.5) [NavyBlue,above] {$c p_i$};
	
	\begin{scope}[xshift=6cm]
	\node at (-3,0) {$= \tfrac{1}{\chi}$};
	\draw[dashed] (0,0) circle (1.8cm);
	\node at (0,-0.5) [NavyBlue] {$\bullet$};
	\draw [NavyBlue] (-0.3,-0.9) -- (0.3,-0.9);
	\draw [thick,NavyBlue] (0,-1.8) -- (0,-0.5);
	\draw [thick,NavyBlue] (-1,-1.5) -- (0,-0.5);
	\node at (0,-.5) [NavyBlue,above] {$c$};
	
	\node at (1,0) [NavyBlue] {$\bullet$};
	\node at (1,0) [NavyBlue,above] {$e p_i$};
	
	\end{scope}

	\end{scope}
	\end{tikzpicture}
\end{figure}
\end{enumerate}
and (d$^{\prime \prime\prime}$).
\end{enumerate}
\end{definition}

\begin{remark}[Graphs and partitions]\label{rem:GraphsAndPartitions} 
In all cases one can apply the (modified) contraction formula to pass from a graph to a sum of graphs with strictly fewer edges, and so by iterating to a sum of graphs with no edges. These are disjoint unions of labelled corollas, and so correspond to partitions of $S$ with labels in $\mathcal{V}$ or $\mathcal{W}$, plus additional external labels in cases (ii) and (iii). There are two issues with this. The first is that in cases (iv) and (v) it is not clear that the resulting sum of disjoint unions of labelled corollas is unique, as one has to choose an order in which to eliminate edges. The second is that even if it is, then the functoriality on the Brauer category which we describe below would involve gluing in edges and then eliminating them, leading to a complicated formula. This is why we have found it convenient to work with spaces of graphs.
\end{remark}

We wish to consider each of the above as defining functors on the (signed) Brauer category as in \cite[Section 2.3]{KR-WTorelli}, but to take into account the parameter $\chi$ we must slightly generalise to a $\bQ[\chi]$-linear version of the (signed) Brauer category.

\begin{definition}
For finite sets $S$ and $T$ let $\mathsf{preBr}^\chi(S,T)$ be the free $\bQ[\chi]$-module on tuples $(f,m_S, m_T)$ of a bijection $f$ from a subset $S^\circ \subset S$ to a subset $T^\circ \subset T$, an ordered matching $m_S$ of $S \setminus S^\circ$, and an ordered matching $m_T$ of $T \setminus T^\circ$.

Let $\mathsf{Br}^\chi(S,T)$ be the quotient of $\mathsf{preBr}^\chi(S,T)$ by the span of $(f,m_S, m_T) - (f,m'_S, m'_T)$ whenever $m_S$ agrees with $m'_S$ after reversing some pairs, and $m_T$ agrees with $m'_T$ after reversing some pairs.

Let $\mathsf{sBr}^\chi(S,T)$ be the quotient of $\mathsf{preBr}^\chi(S,T)$ by the span of $(f,m_S, m_T) - (-1)^{k+l}(f,m'_S, m'_T)$ whenever $m_S$ agrees with $m'_S$ after reversing $k$ pairs, and $m_T$ agrees with $m'_T$ after reversing $l$ pairs.

\tikzset{middlearrow/.style={
        decoration={markings,
            mark= at position 0.5 with {\arrow[scale=2]{#1}} ,
        },
        postaction={decorate}
    }
}

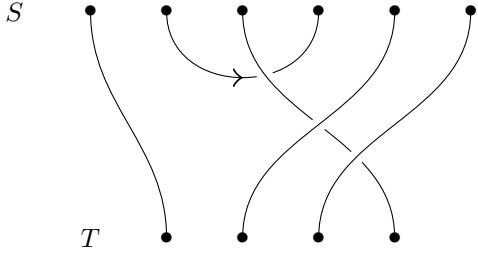
\begin{figure}[h]
	\begin{tikzpicture}
	\draw[middlearrow={to}] (0,3) to[out=-90,in=-90,looseness=1.5] (2,3);
	\draw (-1,3) to[out=-90,in=90] (0,0);
	\draw [line width=2mm,white] (1,3) to[out=-90,in=90] (3,0);
	\draw (1,3) to[out=-90,in=90] (3,0);
	\draw [line width=2mm,white] (3,3) to[out=-90,in=90] (1,0);
	\draw (3,3) to[out=-90,in=90] (1,0);
	\draw [line width=2mm,white] (4,3) to[out=-90,in=90] (2,0);
	\draw (4,3) to[out=-90,in=90] (2,0);
	\foreach \x in {-1,...,4}
	\node at (\x,3) {$\bullet$};
	\foreach \x in {-0,...,3}
	\node at (\x,0) {$\bullet$};
	\node at (-2,3) {$S$};
	\node at (-1,0) {$T$};
	\end{tikzpicture}
	\caption{A graphical representation of a basis element $(f,m_S, m_T)$ in $\mathsf{preBr}^\chi(S,T)$ from a $6$-element set $S$ to a $4$-element set $T$, with $m_T = \emptyset$. (The order of crossings is irrelevant.) Reversing the arrow gives the same element in $\mathsf{Br}^\chi(S,T)$ and minus the element in $\mathsf{sBr}^\chi(S,T)$.}
	\label{fig:brauerdownward}
\end{figure}

Let $\mathsf{(s)Br}^\chi$ be the $\bQ[\chi]$-linear category whose objects are finite sets, and whose morphisms are the $\bQ[\chi]$-modules $\mathsf{(s)Br}^\chi(S,T)$ defined above. In the case of $\mathsf{Br}^\chi$ we think of $[f,m_S, m_T]$ as representing 1-dimensional cobordisms with no closed components: then the composition law is given by composing cobordisms and then replacing each closed 1-manifold by a factor of $\chi-2$. In the case of $\mathsf{sBr}^\chi$ we think of $(f,m_S, m_T)$ as representing oriented 1-dimensional cobordisms with no closed components: then the composition law is given by composing cobordisms and then replacing each compatibly oriented closed 1-manifold by a factor of $-(\chi-2)$.

Let $\mathsf{d(s)Br}^\chi$ denote the subcategories having all objects and morphisms spanned by $[f,m_S, m_T]$ with $T^\circ=\emptyset$.

For a \emph{central charge} $d \in \bQ$ let $\mathsf{(d)(s)Br}_d$ denote the $\bQ$-linear category obtained by specialising the $\bQ[\chi]$-linear category $\mathsf{(d)(s)Br}^\chi$ to $\chi=2 + d$ for $\mathsf{(d)Br}$ or $\chi=2 - d$ for $\mathsf{(d)sBr}$. (This notation then agrees with \cite[Definition 2.14, 2.19]{KR-WTorelli}.)
\end{definition}

We consider the spaces of graphs above as defining $\bQ[\chi]$-linear functors
$$\mathcal{G}\mathrm{raph}_1(-), \mathcal{G}\mathrm{raph}_*^\theta(-), \mathcal{G}\mathrm{raph}_*(-), \mathcal{G}\mathrm{raph}^\theta(-), \mathcal{G}\mathrm{raph}(-) : \mathsf{(s)Br}^\chi \to \mathsf{Gr}(\bQ[\chi^{\pm 1}]\text{-}\mathsf{mod})$$
in the evident way, by gluing of oriented graphs (after orientations have been arranged to be compatible). We endow them with a lax symmetric monoidality by disjoint union of graphs. We write $\mathcal{G}\mathrm{raph}_1(-)_{g} : \mathsf{(s)Br}_{2g} \to \mathsf{Gr}(\bQ\text{-}\mathsf{mod})$  and so on for their specialisations at $\chi=2+(-1)^n2g$ (defined for $(n,g) \neq (\text{odd}, 1)$).

\subsection{The isomorphism theorem}

Theorem \ref{thm:TwistedCoh} below extends \cite[Theorem 3.15]{KR-WTorelli} to $B\mathrm{Diff}^{\theta}(W_g, *)$, $B\mathrm{Diff}^{+}(W_g, *)$, $B\mathrm{Diff}^{\theta}(W_g)$, and $B\mathrm{Diff}^{+}(W_g)$. 

To formulate it we first observe that when $\pi : E \to X$ is a smooth oriented $W_g$-bundle and $\mathcal{H}$ is the local coefficient system over $X$ given by the fibrewise $n$th homology of this bundle, the fibrewise intersection form $\lambda : \mathcal{H} \otimes \mathcal{H} \to \bQ$ and its dual $\omega : \bQ \to \mathcal{H} \otimes \mathcal{H}$ are $(-1)^n$-symmetric and satisfy $\lambda \circ \omega = (-1)^n2g \cdot \mathrm{Id}$, so provide a $\bQ$-linear functor $S \mapsto \mathcal{H}^{\otimes S}$ from $\mathsf{(s)Br}_{2g}$ to the category of local coefficient systems of $\bQ$-modules over $X$. (Strictly speaking our definitions require $\chi = 2 + (-1)^n 2g$ to be invertible, so we omit the case $(n,g) = (\text{odd}, 1)$.) Composing this with taking cohomology gives a functor
\begin{align*}
H^*(X ; \mathcal{H}^{\otimes -}) : \mathsf{(s)Br}_{2g} &\lra \mathsf{Gr}(\bQ\text{-}\mathsf{mod})\\
S &\longmapsto H^*(X ; \mathcal{H}^{\otimes S}).
\end{align*}

The relations in the various spaces of graphs defined in Section \ref{sec:SpacesOfGraphs} were chosen precisely to match the contraction formula of \cite[Proposition 3.10]{KR-WTorelli} (in the case of $\mathcal{G}\mathrm{raph}_1$) and the modified contraction formula of Proposition \ref{prop:MCF} (in the other cases), so that assigning to a graph its associated $\kappa$- or $\bar{\kappa}$-class provides natural transformations
\begin{enumerate}[(i)]
\item $\kappa : \mathcal{G}\mathrm{raph}_1(-)_g  \to H^*(B\mathrm{Diff}(W_g, D^{2n}); \cH^{\otimes -})$,

\item $\kappa: \mathcal{G}\mathrm{raph}_*^\theta(-)_g \to H^*(B\mathrm{Diff}^{\theta}(W_g, *); \cH^{\otimes -})$,

\item $\kappa: \mathcal{G}\mathrm{raph}_*(-)_g \to H^*(B\mathrm{Diff}^{+}(W_g, *); \cH^{\otimes -})$,

\item $\bar{\kappa} : \mathcal{G}\mathrm{raph}^\theta(-)_g \to H^*(B\mathrm{Diff}^{\theta}(W_g); \cH^{\otimes -})$,

\item $\bar{\kappa} : \mathcal{G}\mathrm{raph}(-)_g \to H^*(B\mathrm{Diff}^{+}(W_g); \cH^{\otimes -})$,
\end{enumerate}
of functors $\mathsf{(s)Br}_{2g} \to \mathsf{Gr}(\bQ\text{-}\mathsf{mod})$.

\begin{theorem}\label{thm:TwistedCoh}
For $2n =2$ or $2n \geq 6$ the maps (i)--(v) are isomorphisms in a range of cohomological degrees tending to infinity with $g$.
\end{theorem}

We will first give the proof in cases (i), (ii), (iii), and in case (v) assuming case (iv); the much more involved case (iv) will be treated afterwards.

\begin{proof}[Proof of Theorem \ref{thm:TwistedCoh} (i), (ii), (iii), (v)]
For case (i) observe that $\mathcal{G}\mathrm{raph}_1(-)_g$ is naturally isomorphic to the functor $\mathcal{G}(-,\cV)$ from \cite[Proof of Theorem 5.1]{KR-WTorelli}, which is shown there to be isomorphic to the functor $\mathcal{P}(-, \cV)_{\geq 0} \otimes \det^{\otimes n}$. This case then follows from \cite[Theorem 3.15]{KR-WTorelli}.

For case (ii) we first construct a homotopy fibre sequence
\begin{equation}\label{eq:EvPt}
B\mathrm{Diff}(W_g, D^{2n}) \lra B\mathrm{Diff}^{\theta}(W_g, *) \lra B\SO(2n)\langle n \rangle.
\end{equation}
The left-hand term may be written as the homotopy quotient of $\mathrm{Diff}(W_g, *)$ acting on the Stiefel manifold $\mathrm{Fr}(T_*W_g)$ given by the space of frames in the tangent space to $W_g$ at the point $* \in W_g$, as this action is transitive and its stabiliser is the subgroup which fixes a point and its tangent space, which is homotopy equivalent to fixing a  disc. The middle term was defined as the homotopy quotient of $\mathrm{Diff}(W_g, *)$ acting on $\mathrm{Bun}(TW_g, \theta^*\gamma_{2n})$. Evaluation at $* \in W_g$ defines a $\mathrm{Diff}(W_g, *)$-invariant map
$$ev: \mathrm{Bun}(TW_g, \theta^*\gamma_{2n}) \lra B\SO(2n)\langle n \rangle.$$
This is easily seen to be a Serre fibration. If we choose a point $x \in B\SO(2n)\langle n \rangle$ and a framing $\xi : (\theta^*\gamma_{2n})_x \overset{\sim}\to \bR^{2n}$, then there is a map $\xi_* : ev^{-1}(x) \to \mathrm{Fr}(T_* W_g)$ given by sending a bundle map $\hat{\ell} : TW_g \to \theta^*\gamma_{2n}$ whose underlying map sends $*$ to $x$ to the framing $\xi \circ \hat{\ell}_x : T_* W_g \to (\theta^*\gamma_{2n})_x \to \bR^{2n}$. One verifies by obstruction theory that $\xi_* : ev^{-1}(x) \to \mathrm{Fr}(T_* W_g)$ is a weak equivalence. Taking homotopy orbits for $\mathrm{Diff}(W_g, *)$ then gives the required homotopy fibre sequence \eqref{eq:EvPt}.

As $H^*(B\mathrm{Diff}(W_g, D^{2n}); \cH^{\otimes S})$ is spanned by products of twisted Miller--Morita--Mumford classes $\kappa_{\epsilon^a c}$ with $c \in \mathcal{V}$ in a stable range by (i), and these classes may be defined on $B\mathrm{Diff}^{\theta}(W_g, *)$, the Serre spectral sequence 
$$H^*(B\SO(2n)\langle n \rangle;\bQ) \otimes H^*(B\mathrm{Diff}(W_g, D^{2n}); \cH^{\otimes S}) \Rightarrow H^*(B\mathrm{Diff}^{\theta}(W_g, *); \cH^{\otimes S})$$
for the homotopy fibre sequence \eqref{eq:EvPt} collapses in a stable range. The result then follows by observing that the analogue of the Serre filtration of $\mathcal{G}\mathrm{raph}_*^\theta(-)_g$, induced by the descending filtration by degrees of $H^*(B\SO(2n)\langle n \rangle;\bQ) = \cV$, has
$$\mathrm{gr}(\mathcal{G}\mathrm{raph}_*^\theta(-)_g) \cong \cV \otimes \mathcal{G}\mathrm{raph}_1(-)_g,$$
because modulo $\mathcal{V}_{>0}$ the formula of (d$^\prime$) specialises to that of (d). The induced map
$$\mathrm{gr}(\kappa) : \mathrm{gr}(\mathcal{G}\mathrm{raph}_*^\theta(-)_g) \lra \mathrm{gr}(H^*(B\mathrm{Diff}^{\theta}(W_g, *); \cH^{\otimes -}))$$
therefore has the form $\mathcal{V} \otimes \{\text{the map $\kappa$ in case (i)}\}$ so is an isomorphism in a stable range by case (i). Case (ii) follows.

Case (iii) is just like the above, using the homotopy fibre sequence
$$B\mathrm{Diff}(W_g, D^{2n}) \lra B\mathrm{Diff}^{+}(W_g, *) \lra B\SO(2n)$$
instead, which is established in the analogous way, and $\cW$ in place of $\cV$.

Case (v) can be deduced from case (iv) by applying the same method to the homotopy fibre sequence
$$B\mathrm{Diff}^\theta(W_g) \lra B\mathrm{Diff}^{+}(W_g) \overset{\xi}\lra B\SO[0,n]$$
established in \cite[Section 5.2]{grwsurvey}. The filtration step is a little different, so we give some details. It follows from (iv) that $H^*(B\mathrm{Diff}^\theta(W_g); \cH^{\otimes S})$ is spanned by products of twisted Miller--Morita--Mumford classes $\kappa_{\bar{\epsilon}^a c}$ with $c \in \mathcal{V}$ in a stable range, and these may be defined on $B\mathrm{Diff}^{+}(W_g)$ (in fact they may be defined even for $c \in \mathcal{W}$) so the corresponding Serre spectral sequence 
$$H^*(B\SO[0,n];\bQ) \otimes H^*(B\mathrm{Diff}^\theta(W_g); \cH^{\otimes S}) \Rightarrow H^*(B\mathrm{Diff}^{+}(W_g); \cH^{\otimes S})$$
degenerates in a stable range. In this case the analogue of the Serre filtration on $\mathcal{G}\mathrm{raph}(-)_g$ is induced by giving the graph $\Upsilon_i := (\{0\}, \emptyset, \emptyset \to \{0\}, \emptyset, c(0) = ep_i)$---which is a single vertex labelled by $ep_i$ and no edges---filtration $4i$ for $1 \leq i  \leq \lfloor n/4 \rfloor$, giving all other connected graphs filtration 0, and extending multiplicatively. The associated graded of this filtration has the form
$$\mathrm{gr}(\mathcal{G}\mathrm{raph}(-)_g) \cong \bQ[\Upsilon_1, \Upsilon_2, \ldots, \Upsilon_{ \lfloor n/4 \rfloor}] \otimes \mathcal{G}\mathrm{raph}^\theta(-)_g,$$
because the relation in (c$^{\prime\prime\prime\prime}$) shows that any graph with a vertex labelled $cp_i$ for $1 \leq i  \leq \lfloor n/4 \rfloor$ is equivalent to a graph of strictly larger filtration, unless the vertex is 0-valent and the label is $ep_i$. As $\bar{\kappa}(\Upsilon_i) = \kappa_{ep_i} = \chi \cdot \xi^*(p_i)$ by \cite[Remark 5.5]{grwsurvey} it follows that the induced map
$$\mathrm{gr}(\bar{\kappa}) : \mathrm{gr}(\mathcal{G}\mathrm{raph}(-)_g) \lra \mathrm{gr}(H^*(B\mathrm{Diff}^{+}(W_g); \cH^{\otimes S}))$$
has the form $\{\text{an isomorphism}\} \otimes \{\text{the map $\bar{\kappa}$ in case (iv)}\}$ so is an isomorphism in a stable range by case (iv).
\end{proof}

\subsection{Proof of Theorem \ref{thm:TwistedCoh} (iv)}

The proof of Theorem \ref{thm:TwistedCoh} (iv) is of a less formal nature. It will be parallel to that of \cite[Theorem 3.15]{KR-WTorelli}, but algebraically more complicated. An important tool will be the following lemma, inspired by \cite[p.\ 566]{QuillenEqI}.

\begin{lemma}\label{lem:QuillenLemma}
Let $G$ be a topological group and $p : P \to X$ be a principal $G$-bundle with action $a : G \times P \to P$, which satisfies the Leray--Hirsch property in cohomology over a field $\bF$. Then
\begin{equation*}
\begin{tikzcd}
H^*(X;\bF) \rar{p^*} & H^*(P;\bF) \ar[r,shift left=.75ex,"a^*"]
  \ar[r,shift right=.75ex,swap,"1 \otimes \mathrm{Id}"]& H^*(G;\bF) \otimes_\bF H^*(P;\bF)
\end{tikzcd}
\end{equation*}
is an equaliser diagram.
\end{lemma}
\begin{proof}
Let us leave $\bF$ implicit. By the Leray--Hirsch property $H^*(P)$ is a free $H^*(X)$-module and hence is faithfully flat. Thus it suffices to prove that the diagram is an equaliser diagram after applying $- \otimes_{H^*(X)} H^*(P)$. By the Leray--Hirsch property we also have $H^*(P) \otimes_{H^*(X)} H^*(P) \overset{\sim}\to H^*(P \times_X P)$. Thus it suffices to show that
\begin{equation*}
\begin{tikzcd}
H^*(P) \rar{\mathrm{pr}_2^*} & H^*(P \times_X P) \ar[r,shift left=.75ex,"a^*"]
  \ar[r,shift right=.75ex,swap,"1 \otimes \mathrm{Id}"]& H^*(G) \otimes H^*(P \times_X P)
\end{tikzcd}
\end{equation*}
is an equaliser diagram, which is the same question for the principal $G$-bundle $\mathrm{pr}_2 : P \times_X P \to P$. But this principal $G$-bundle has a section given by the diagonal map, which trivialises it: this trivialisation identifies the diagram with
\begin{equation*}
\begin{tikzcd}
H^*(P) \rar{1 \otimes \mathrm{Id}} & H^*(G) \otimes H^*(P) \ar[r,shift left=.75ex,"\mu^* \otimes \mathrm{Id}"]
  \ar[r,shift right=.75ex,swap,"1 \otimes\mathrm{Id}"]& H^*(G) \otimes H^*(G) \otimes H^*(P)
\end{tikzcd}
\end{equation*}
which is indeed an equaliser diagram as it has a contraction induced by $a^*$.
\end{proof}

We adapt the proof of \cite[Theorem 3.15]{KR-WTorelli}, supposing for concreteness that $n$ is odd. Consider the tangential structure $\theta \times Y : B\SO(2n) \langle n \rangle \times Y \to B\SO(2n)$ with $Y=K(W^\vee, n+1)$ and $W$ a generic rational vector space. Then we have $H^*(Y ; \bQ) \cong \mathrm{Sym}^*(W[n+1])$, the symmetric algebra on the vector space $W$ places in (even) degree $n+1$. If $n$ is even then like at the end of the proof of \cite[Theorem 3.15]{KR-WTorelli} we would take $Y=K(W^\vee, n+2)$ instead, so $H^*(Y;\bQ)$ would still be a symmetric algebra. Apart from this there is no essential difference, and we will not comment further on the differences in the case $n$ even.

There are associated universal $W_g$-bundles
$$\pi: E^{\theta} \lra B\Diff^{\theta}(W_g)\quad\quad\quad\quad
\pi^Y : E^{\theta \times Y} \lra B\Diff^{\theta \times Y}(W_g)$$
and an evaluation map $\ell : E^{\theta \times Y} \to Y$. Neglecting the ``maps to $Y$'' part of the tangential structure gives a homotopy fibre sequence
$$\map(W_g, Y) \lra B\Diff^{\theta \times Y}(W_g) \lra B\Diff^{\theta}(W_g).$$
We can take $Y$ to be a topological abelian group, which then acts fibrewise on the map $\theta \times Y$ and hence acts compatibly on $E^{\theta \times Y}$ and $B\Diff^{\theta \times Y}(W_g)$. Using this we can form the homotopy fibre sequence
$$\map(W_g, Y) \hcoker Y \lra B\Diff^{\theta \times Y}(W_g) \hcoker Y \lra B\Diff^{\theta}(W_g).$$
The space $\map(W_g, Y) \hcoker Y$ is a $K(H^n(W_g;\bQ) \otimes W^\vee,1)$, so there is an identification of graded local coefficient systems
$$\cH^*(\map(W_g, Y) \hcoker Y;\bQ) = \Lambda^*(\cH \otimes W[1]).$$
This is natural in the vector space $W$, and scaling by $u \in \bQ^\times$ acts on $\Lambda^k(\cH \otimes W[1])$ by $u^k$. It follows that it acts this way on the $k$th row of the Serre spectral sequence
$$E_2^{p,q} = H^p(B\Diff^{\theta}(W_g); \Lambda^q(\cH \otimes W[1])) \Rightarrow H^{p+q}(B\Diff^{\theta \times Y}(W_g) \hcoker Y;\bQ).$$
As the differentials in this spectral sequence must be equivariant for this $\bQ^\times$-action, it follows that they must all be trivial. Furthermore this action gives a weight decomposition of both sides, which identifies
$$H^*(B\Diff^{\theta}(W_g); \Lambda^k(\cH \otimes W)) \cong H^{*+k}(B\Diff^{\theta \times Y}(W_g) \hcoker Y;\bQ)^{(k)},$$
the weight $k$-subspace.

To access the latter groups, we use that there is a map
$$\alpha: B\Diff^{\theta \times Y}(W_g) \lra \Omega^\infty_0(\mathrm{MT}\theta \wedge Y_+)$$
which by the main theorems of \cite{Boldsen, oscarresolutions, GMTW} (for $2n=2$) and \cite{grwstab1, grwcob, grwstab2} for ($2n \geq 6$) is an isomorphism on cohomology in a stable range of degrees. Here $\mathrm{MT}\theta$ is the Thom spectrum of $-\theta^*\gamma_{2n}$, so writing $u_{-2n} \in H^{-2n}(\mathrm{MT}\theta;\bQ)$ for its Thom class, by the Thom isomorphism we have
$$H^*(\mathrm{MT}\theta;\bQ) \cong u_{-2n} \cdot H^*(B\SO(2n) \langle n \rangle ;\bQ) = u_{-2n}\cdot \bQ[p_{\lceil \tfrac{n+1}{4}\rceil}, \ldots, p_{n-1}, e].$$
The rational cohomology of $\Omega^\infty_0(\mathrm{MT}\theta \wedge Y_+)$ is then given by 
$$\mathrm{Sym}^*([H^*(\mathrm{MT}\theta;\bQ) \otimes \mathrm{Sym}^*(W[n+1])]_{>0}),$$
which can be considered as the free (graded-)commutative algebra on the even-degree classes $\kappa_{c, w_1 \cdots w_r}$ with $c \in  \bQ[p_{\lceil \tfrac{n+1}{4}\rceil}, \ldots, p_{n-1}, e]$ and $w_i \in W$, modulo linearity in $c$ and in the $w_i$, and modulo commutativity of the $w_i$. The pullbacks of these classes along $\alpha$ we again denote $\kappa_{c, w_1 \cdots w_r}$, and they may be described intrinsically as the fibre integrals $\pi^Y_!(c(T_{\pi^Y} E^{\theta \times Y}) \cdot \ell^*(w_1 \cdots w_r))$. 

\begin{lemma}
 There are unique classes $\bar{\kappa}_{c, w_1 \cdots w_r} \in H^*(B\Diff^{\theta \times Y}(W_g) \hcoker Y;\bQ)$ which pull back to
$$\sum_{I \sqcup J = \{1,2,\ldots,r\}} \kappa_{c, w_I} \cdot \prod_{j \in J} (-\tfrac{1}{\chi} \kappa_{e, w_j}) \in H^*(B\Diff^{\theta \times Y}(W_g);\bQ),$$
and in a stable range of degrees $H^*(B\Diff^{\theta \times Y}(W_g) \hcoker Y;\bQ)$ is the free graded-commutative algebra on the classes $\bar{\kappa}_{c, w_1 \cdots w_r}$, modulo linearity in $c$ and in the $w_i$, commutativity of the $w_i$, and modulo $\bar{\kappa}_{e, w_1}=0$. 
\end{lemma}

\begin{proof}
We wish to apply Lemma \ref{lem:QuillenLemma} to the principal $Y$-bundle
\begin{equation}\label{eq:YBun}
B\Diff^{\theta \times Y}(W_g) \lra B\Diff^{\theta \times Y}(W_g) \hcoker Y.
\end{equation}
First observe that the fibre inclusion $j : Y \to B\Diff^{\theta \times Y}(W_g)$ classifies the $W_g$-bundle $\mathrm{pr}_1: Y \times W_g \to Y$ equipped with the product $\theta$-structure and the map $\ell = \mathrm{pr}_1 : Y \times W_g \to Y$. Thus for any $w \in W$ we have
$$j^* \kappa_{e, w} = \chi w \in H^{n+1}(Y;\bQ),$$
and so \eqref{eq:YBun} satisfies the Leray--Hirsch property. 

Lemma \ref{lem:QuillenLemma} then describes $H^*(B\Diff^{\theta \times Y}(W_g) \hcoker Y;\bQ)$ as the equaliser of
\begin{equation}\label{eq:EqualiserFormula}
\begin{tikzcd}
 H^*(B\Diff^{\theta \times Y}(W_g);\bQ) \ar[r,shift left=.75ex,"a^*"]
  \ar[r,shift right=.75ex,swap,"1 \otimes \mathrm{Id}"]& H^*(Y;\bQ) \otimes H^*(B\Diff^{\theta \times Y}(W_g);\bQ).
\end{tikzcd}
\end{equation}
In a stable range $H^*(B\Diff^{\theta \times Y}(W_g);\bQ)$ is described in terms of the classes $\kappa_{c, w_1 \cdots w_r}$, so to make use of this equaliser description we must determine how these classes pull back along the action map
$$a : Y \times B\Diff^{\theta \times Y}(W_g) \lra B\Diff^{\theta \times Y}(W_g).$$
This map classifies the $W_g$-bundle $Y \times \pi^Y : Y \times E^{\theta \times Y} \to Y \times B\Diff^{\theta \times Y}(W_g)$ equipped with the structure map $Y \times E^{\theta \times Y} \overset{Y \times \ell}\to Y \times Y \overset{\cdot}\to Y$. As the $w_i \in W = H^{n+1}(Y;\bQ)$ are primitive with respect to the coproduct induced by the multiplication on $Y$, we have
\begin{align*}
a^*(\kappa_{c, w_1 \cdots w_r}) &= (Y \times \pi^Y)_!((1 \times c(T_{\pi^Y} E^{\theta \times Y})) \cdot \prod_{i=1}^r (w_i \times 1 + 1 \times \ell^*(w_i)))\\
&=\sum_{I \sqcup J = \{1,2,\ldots,r\}} w_I \times \kappa_{c, w_J}.
\end{align*}

Our goal now is to show that the classes defined by
$$\bar{\kappa}_{c, w_1 \cdots w_r} := \sum_{I \sqcup J = \{1,2,\ldots,r\}} \kappa_{c, w_I} \cdot \prod_{j \in J} (-\tfrac{1}{\chi} \kappa_{e, w_j}) \in  H^*(B\Diff^{\theta \times Y}(W_g);\bQ)$$
are equalised by the maps \eqref{eq:EqualiserFormula}, so by Lemma \ref{lem:QuillenLemma} descend to unique classes of the same name in $H^*(B\Diff^{\theta \times Y}(W_g) \hcoker Y;\bQ)$. To see this, we calculate using the formula above that
\begin{align*}
a^*(\bar{\kappa}_{c, w_1 \cdots w_r}) &= \sum_{I \sqcup J = \{1,2,\ldots,r\}} \left(\sum_{S \sqcup T = I} w_S \times \kappa_{c, w_T}\right) \cdot (-1)^{|J|} \prod_{j \in J} (w_j \times 1 + \tfrac{1}{\chi} 1 \times \kappa_{e, w_j})\\
&= \sum_{S \sqcup T \sqcup U \sqcup V = \{1,2,\ldots,r\}} (-1)^{|U|} w_{S \sqcup U} \times \left(\kappa_{c, w_T} \cdot \prod_{v \in V} (-\tfrac{1}{\chi} \kappa_{e, w_v})\right).
\end{align*}
For each $A \subseteq \{1,2,\ldots, r\}$ the coefficient of $w_A$ is
$$ \left(\sum_{U \subseteq A}(-1)^{|U|}\right) \left( \sum_{T \sqcup V = \{1,\ldots,r\} \setminus A} \kappa_{c, w_T} \cdot \prod_{v \in V} (-\tfrac{1}{\chi} \kappa_{e, w_v})\right),$$
and $\sum_{U \subseteq A}(-1)^{|U|}$ is the binomial expansion of $(1-1)^{|A|}$ so vanishes if $A \neq \emptyset$ and is $1$ if $A=\emptyset$, which shows that $a^*(\bar{\kappa}_{c, w_1 \cdots w_r}) = 1 \times \bar{\kappa}_{c, w_1 \cdots w_r}$ as required.

Finally, that these classes (except $\bar{\kappa}_{e, w_1}=0$) freely generate the $\bQ$-algebra $H^*(B\Diff^{\theta \times Y}(W_g) \hcoker Y;\bQ)$ in a stable range follows from the fact that the $\kappa_{c, w_1 \cdots w_r}$ freely generate $H^*(B\Diff^{\theta \times Y}(W_g);\bQ)$ in a stable range, together with the observation that $\bar{\kappa}_{c, w_1 \cdots w_r} \equiv {\kappa}_{c, w_1 \cdots w_r}$ modulo the ideal generated by classes $\kappa_{e, w}$ and the Leray--Hirsch property again.
\end{proof}

Let us provide a ``fibre-integral" interpretation of the classes we have just constructed. Consider the map of principal $Y$-bundles
\begin{equation*}
\begin{tikzcd}
Y \rar{i} \arrow[d, equal] & E^{\theta \times Y} \rar \dar{\pi^Y} & E^{\theta \times Y}  \hcoker Y \dar{\pi^Y \hcoker Y}\\
Y \rar{j} & B\Diff^{\theta \times Y}(W_g) \rar & B\Diff^{\theta \times Y}(W_g)  \hcoker Y.
\end{tikzcd}
\end{equation*}
The composition $\ell \circ i : Y \to Y$ is the identity, so $i^* \ell^*(w) = w \in H^{n+1}(Y;\bQ)$. We showed in the proof above that $j^* \kappa_{e, w} = \chi w \in H^{n+1}(Y;\bQ)$, so in particular both these principal $Y$-bundles satisfy the Leray--Hirsch property. Together these give that 
$$i^*(\ell^*(w) - \tfrac{1}{\chi} (\pi^Y)^*\kappa_{e, w})=0.$$ 
As $Y$ is $n$-connected it follows from the Serre spectral sequence that there exists a unique class $\bar{\ell}^*(w) \in H^{n+1}(E^{\theta \times Y}  \hcoker Y;\bQ)$ which pulls back to $\ell^*(w) - \tfrac{1}{\chi} (\pi^Y)^*\kappa_{e, w}$. 

\begin{lemma}\label{lem:KappaBarAsIntegral}
We have
\begin{equation*}
\bar{\kappa}_{c, w_1 \cdots w_r} = (\pi^Y \hcoker Y)_!(c \cdot \bar{\ell}^*(w_1) \cdots \bar{\ell}^*(w_r)) \in H^*(B\Diff^{\theta \times Y}(W_g) \hcoker Y;\bQ).
\end{equation*}
\end{lemma}
\begin{proof}
As the lower of the above principal $Y$-bundles satisfies the Leray--Hirsch property, this identity may be verified after pulling back to $B\Diff^{\theta \times Y}(W_g)$. In $H^*(E^{\theta \times Y};\bQ)$ we have $\bar{\ell}^*(w)  =\ell^*(w) - \tfrac{1}{\chi} (\pi^Y)^*\kappa_{e, w}$, so expanding out gives
\begin{align*}
(\pi^Y)_!(c \cdot \bar{\ell}^*(w_1) \cdots \bar{\ell}^*(w_r)) &= (\pi^Y)_!(c \cdot \prod_{i=1}^r (\ell^*(w_i) - \tfrac{1}{\chi} (\pi^Y)^*\kappa_{e, w_i}))\\
&= \sum_{I \sqcup J = \{1,2,\ldots,r\}} \kappa_{c, w_I} \cdot \prod_{j \in J} (-\tfrac{1}{\chi} \kappa_{e, w_j})
\end{align*}
as required.
\end{proof}

The classes $\bar{\kappa}_{c, w_1 \cdots w_r}$ provide an isomorphism
$$\mathrm{Sym}^*\left(\frac{[H^*(\mathrm{MT}\theta;\bQ) \otimes \mathrm{Sym}^*(W[n+1])]_{>0}}{u_{-2n} \cdot e \otimes W[n+1]}\right) \lra H^*(B\Diff^{\theta \times Y}(W_g) \hcoker Y ; \bQ)$$
in a stable range, natural in $W$, which with the discussion above gives an identification of graded vector spaces
$$H^*(B\Diff^{\theta}(W_g); \Lambda^*(\cH \otimes W[1])) \cong \mathrm{Sym}^*\left(\frac{[H^*(\mathrm{MT}\theta;\bQ) \otimes \mathrm{Sym}^*(W[n+1])]_{>0}}{u_{-2n} \cdot e \otimes W[n+1]}\right)$$
natural in $W$. 

We now proceed as in the proof of \cite[Theorem 3.15]{KR-WTorelli}, which we briefly summarise. As the identification just given is natural in $W$, we can consider it as an identification of (graded) Schur functors of $W$, and therefore identify their coefficients (i.e.\ Schur-Weyl duality). This is implemented by letting $\rm{GL} := \colim_{n \to \infty} \rm{GL}_n(\bQ)$ and considering the two sides as defining (graded) polynomial representations of $\rm{GL}$ by setting $W := \colim_{n \to \infty} \bQ^n$. Writing $\hat{W}_* := \lim_{n \to \infty} \mathrm{Hom}(\bQ^n, \bQ)$, Schur--Weyl duality may be expressed as the equivalence of categories
$$
 \begin{tikzcd}[every arrow/.append style={shift left}]
 (\bQ\text{-}\mathsf{mod}^{\mathsf{FB}})^f \arrow{rrr}{F \mapsto W^{\otimes -} \otimes^{\mathsf{FB}} F} &&& \mathsf{Rep}^{\mathsf{pol}}(\rm{GL}), \arrow{lll}{[S \mapsto [\hat{W}_*^{\otimes S} \otimes U]^{\rm{GL}}] \mapsfrom U}    
 \end{tikzcd}
$$
where $(\bQ\text{-}\mathsf{mod}^{\mathsf{FB}})^f$ denotes the functors from the category $\mathsf{FB}$ of finite sets and bijections to the category of finite-dimensional $\bQ$-modules; it gives a similar equivalence between categories of graded objects. Using that $n$ is odd, the left-hand side above can be expressed as
$$H^*(B\Diff^{\theta}(W_g); \Lambda^*(\cH \otimes W[1])) \cong W^{\otimes -} \otimes^{\mathsf{FB}} \left(H^*(B\Diff^{\theta}(W_g); \cH^{\otimes -}) \otimes {\det}^{\otimes n} \right),$$
up to a small reindexing of degrees. It is elementary to see that the right-hand side may be expressed as $W^{\otimes -} \otimes^{\mathsf{FB}} \mathcal{P}^\text{bis}(-,\cV)_{\geq 0}$ up to the analogous reindexing of degrees, where $ \mathcal{P}(-,\cV)_{\geq 0} \to \mathcal{P}^\text{bis}(-,\cV)_{\geq 0}$ is the quotient by those partitions containing a part of size 1 labelled by $e \in \cV_{2n}$. Thus, carrying the factor of ${\det}^{\otimes n}$ to the other side, the above identification yields a natural transformation
\begin{equation}\label{eq:MapFromPart}
\mathcal{P}^\text{bis}(-,\cV)_{\geq 0} \otimes {\det}^{\otimes n} \lra H^*(B\Diff^{\theta}(W_g); \cH^{\otimes -})
\end{equation}
of lax symmetric monoidal functors $\mathsf{FB} \to \mathsf{Gr}(\bQ\text{-}\mathsf{mod})$ which is an isomorphism in a stable range.

On the other hand assigning to a labelled part the corolla with that label gives a natural transformation
\begin{equation}\label{eq:MapToGraph}
\mathcal{P}^\text{bis}(-,\cV)_{\geq 0} \otimes {\det}^{\otimes n} \lra \mathcal{G}\mathrm{raph}^\theta(-)_g,
\end{equation}
of lax symmetric monoidal functors $\mathsf{FB} \to \mathsf{Gr}(\bQ\text{-}\mathsf{mod})$, and we claim that using this \eqref{eq:MapFromPart} factors through the map $\bar{\kappa} : \mathcal{G}\mathrm{raph}^\theta(-)_g \to H^*(B\Diff^{\theta}(W_g); \cH^{\otimes -})$. Assuming this claim for now, observe that using the contraction relations in Definition \ref{defn:GraphSpaces} (iv) (d$^{\prime\prime\prime}$) to contract all edges shows that \eqref{eq:MapToGraph} is surjective, which with the fact that \eqref{eq:MapFromPart} is an isomorphism in a stable range will show that the map $\bar{\kappa}$ is an isomorphism in a stable range too (as well as the map \eqref{eq:MapToGraph}).

It remains to show the factorisation, i.e.\ that the map \eqref{eq:MapFromPart} sends a part of size $a$ labelled by $c \in \cV$ to the class $\kappa_{\bar{\epsilon}^a c}$. We again proceed as in the relevant step of the proof of \cite[Theorem 3.15]{KR-WTorelli}. There is a fibration sequence
$$\map(W_g, Y) \lra E^{\theta \times Y} \lra E^{\theta}$$
and so, taking homotopy orbits for the fibrewise $Y$-action, a fibration sequence
$$\map(W_g, Y) \hcoker Y \lra E^{\theta \times Y} \hcoker Y \lra E^{\theta}.$$
Again by functoriality in $W$ the associated Serre spectral sequence collapses to identify the weight decomposition as
$$H^*(E^{\theta} ; \Lambda^k(\cH \otimes W)) \cong H^{*+k}(E^{\theta \times Y} \hcoker Y;\bQ)^{(k)}.$$
Given the description in Lemma \ref{lem:KappaBarAsIntegral} we must show that the map
$$\bar{\ell}(-) : W \lra H^{n+1}(E^{\theta \times Y}  \hcoker Y;\bQ)^{(1)} \cong H^n(E^\theta ; \cH) \otimes W$$
is given by $w \mapsto \bar{\epsilon} \otimes w$, which is the analogue of \cite[Claim 3.16]{KR-WTorelli}. As it is natural in the vector space $W$ it must certainly be given by $\bar{\ell}(w) = x \otimes w$ for some $x \in H^n(E^\theta ; \cH)$, and we must show that $x = \bar{\epsilon}$. That the restriction of $x$ to the fibre $W_g$ of $\pi : E^\theta \to B\Diff^{\theta}(W_g)$ is given by coevaluation may be done precisely as in \cite[Claim 3.16]{KR-WTorelli}. By the characterisation of $\bar{\epsilon}$ it remains to check that
$$\tfrac{1}{\chi}(\pi^Y)_!(e \cdot \bar{\ell}^*(w)) = 0 \in H^{n+1}(B\Diff^{\theta \times Y}(W_g) \hcoker Y;\bQ).$$
By the Leray--Hirsch property this may be checked after pulling back to $B\Diff^{\theta \times Y}(W_g)$, but as $\bar{\ell}^*(w) = \ell^*(w) - \tfrac{1}{\chi} (\pi^Y)^*\kappa_{e, w} \in H^{n+1}(E^{\theta \times Y};\bQ)$ by definition, the vanishing is immediate.

\subsection{Comparisons}

There are natural maps
\begin{equation*}
\begin{tikzcd}
B\mathrm{Diff}(W_g, D^{2n}) \rar{a} \arrow[d, equals]& B\mathrm{Diff}^\theta(W_g, *) \rar{b} \dar{c}& B\mathrm{Diff}^\theta(W_g) \dar{d}\\
B\mathrm{Diff}(W_g, D^{2n}) \rar{e}& B\mathrm{Diff}^{+}(W_g, *) \rar{f}& B\mathrm{Diff}^{+}(W_g)
\end{tikzcd}
\end{equation*}
which each induce maps on $H^*(-; \mathcal{H}^{\otimes S})$. There are corresponding maps of spaces of graphs
\begin{equation*}
\begin{tikzcd}
\mathcal{G}\mathrm{raph}_1(-)_g  \arrow[d, equals]& \mathcal{G}\mathrm{raph}_*^\theta(-)_g \arrow[l, swap, "a^*"]& \mathcal{G}\mathrm{raph}^\theta(-)_g \arrow[l, swap, "b^*"]\\
\mathcal{G}\mathrm{raph}_1(-)_g & \mathcal{G}\mathrm{raph}_*(-)_g \arrow[l, swap, "e^*"] \arrow[u, swap, "c^*"]& \mathcal{G} \mathrm{raph}(-)_g \arrow[u, swap, "d^*"] \arrow[l, swap, "f^*"]
\end{tikzcd}
\end{equation*}
given as follows. The maps $c^*$ and $d^*$ are induced by the projections $\mathcal{W} \to \mathcal{V}$. The maps $a^*$ and $e^*$ are induced by applying the augmentations $\mathcal{V} \to \bQ$ and $\mathcal{W} \to \bQ$ to the second tensor factor. The maps $b^*$ and $f^*$ are more subtle, as they involve converting between \textcolor{NavyBlue}{blue} graphs and \textcolor{Mahogany}{red} graphs, via the formula of Proposition \ref{prop:Pullback}. Graphically it is given by 
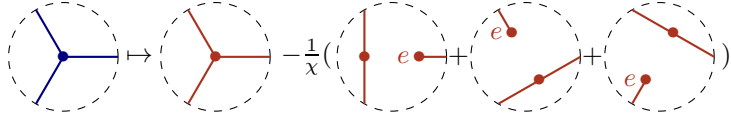
\begin{figure}[h]
	\begin{tikzpicture}[arrowmark/.style 2 args={decoration={markings,mark=at position #1 with \arrow{#2}}}]
	\begin{scope}[scale=0.4]
	\draw[dashed] (0,0) circle (1.8cm);
	\node at (0,0) [NavyBlue] {$\bullet$};
	\foreach \i in {1,...,3}
	{
		\draw [thick,NavyBlue] (0,0) -- ({360/3*\i}:1.8);
	}

	\node at (2.5,0) {$\mapsto$};
	
	\begin{scope}[xshift=5cm]
	\draw[dashed] (0,0) circle (1.8cm);
	\node at (0,0) [Mahogany] {$\bullet$};
	\foreach \i in {1,...,3}
	{
		\draw [thick,Mahogany] (0,0) -- ({360/3*\i}:1.8);
	}
	\end{scope}
	
	\node at (8,0) {$- \frac{1}{\chi} ($};
	
	\begin{scope}[xshift=10.8cm]
	\foreach \j in {1,...,3}{
		\draw[dashed] ({(4.4)*(\j-1)},0) circle (1.8cm);
		\draw [thick,Mahogany] ++({(4.4)*(\j-1)},0) +({360/3*(\j)}:1.8) -- +({360/3*(\j+1)}:1.8);
		\draw ++({(4.4)*(\j-1)},0) +({360/3*(\j+2)}:-.9) node [Mahogany] {$\bullet$};
	 	\draw [thick,Mahogany] ++({(4.4)*(\j-1)},0) +({360/3*(\j+2)}:1.8) -- +({360/3*(\j+2)}:.9);
	 	\draw ++({(4.4)*(\j-1)},0) +({360/3*(\j+2)}:.9) node [Mahogany] {$\bullet$};
		\draw ++({(4.4)*(\j-1)},0) +({360/3*(\j+2)}:.9) node [Mahogany, left] {$e$};
    	}
	\node at (2.2,0) {$+$};
	\node at (6.6,0) {$+$};
	\node at (11,0) {$)$};
	\end{scope}
	\end{scope}
	\end{tikzpicture}
\end{figure}

\noindent with certain orderings.

The maps $b$ and $f$ are also oriented $W_g$-bundles, so they also induce fibre-integration maps $b_!$ and $f_!$ on cohomology. These are $b^*$- and $f^*$-linear respectively, so are determined by the maps (of degree $-2n$)
$$b_! : \mathcal{V} \lra \mathcal{G}\mathrm{raph}^\theta(-)_g \quad\quad f_! : \mathcal{W} \lra \mathcal{G}\mathrm{raph}(-)_g$$
which each send a monomial $c$ in $p_i$'s and $e$ to the graph given by a single vertex labelled by $c$.

\section{Cohomology of Torelli groups}

The isomorphisms provided by Theorem \ref{thm:TwistedCoh} can be converted into information about the spaces
$$B\mathrm{Tor}(W_g, D^{2n}), B\mathrm{Tor}^\theta(W_g, *), B\mathrm{Tor}^{+}(W_g, *), B\mathrm{Tor}^\theta(W_g), B\mathrm{Tor}^{+}(W_g)$$
just as \cite[Theorem 4.1]{KR-WAlg} is deduced from  \cite[Theorem 3.15]{KR-WAlg}. Let us give the definition of these spaces and formulate the result: the following is largely a reminder of some points from \cite{KR-WAlg}, and we do not spell out all details again.

The group $\mathrm{Diff}^{+}(W_g)$ acts on $H_n(W_g;\bZ)$ preserving the nondegenerate $(-1)^n$-symmetric intersection form $\lambda : H_n(W_g;\bZ) \otimes H_n(W_g;\bZ) \to \bZ$. This provides a homomorphism
$$\alpha_g : \mathrm{Diff}^{+}(W_g) \lra G_g := \begin{cases}
\mathrm{Sp}_{2g}(\bZ) & \text{ if $n$ is odd,}\\
\mathrm{O}_{g,g}(\bZ) &  \text{ if $n$ is even.}
\end{cases}$$
This map is not always surjective, but its image is a certain finite-index subgroup $G'_g \leq G_g$, an arithmetic group associated to the algebraic group $\mathbf{Sp}_{2g}$ or $\mathbf{O}_{g,g}$. This subgroup has been determined by Kreck \cite{kreckisotopy}: it is the whole of $G_g$ if $n$ is even or $n=1,3,7$, and otherwise is the subgroup $\mathrm{Sp}_{2g}^q(\bZ) \leq \mathrm{Sp}_{2g}(\bZ)$ of those matrices which preserve the standard quadratic refinement (of Arf invariant 0).

We define $\mathrm{Tor}^{+}(W_g)$ to be the kernel of this homomorphism, and $\mathrm{Tor}^{+}(W_g, *)$ and $\mathrm{Tor}(W_g, D^{2n})$ to be the kernel of its restriction to the subgroups $\mathrm{Diff}^{+}(W_g, *)$ and $\mathrm{Diff}(W_g, D^{2n})$ respectively (these restrictions still have image $G'_g$). Furthermore, we define
\begin{align*}
B\mathrm{Tor}^\theta(W_g) &:= \mathrm{Bun}^+(TW_g, \theta^*\gamma_{2n}) \hcoker \mathrm{Tor}^+(W_g)\\
B\mathrm{Tor}^\theta(W_g, *) &:= \mathrm{Bun}^+(TW_g, \theta^*\gamma_{2n}) \hcoker \mathrm{Tor}^+(W_g, *),
\end{align*}
where $\mathrm{Bun}^+(TW_g, \theta^*\gamma_{2n}) \subset \mathrm{Bun}(TW_g, \theta^*\gamma_{2n})$ consists of the orientation-preserving bundle maps (for some choice of orientation of $\theta^*\gamma_{2n}$ that we make once and for all). By the discussion at the beginning of Section \ref{sec:TwistedCoh} the spaces $\mathrm{Bun}^+(TW_g, \theta^*\gamma_{2n})$ are path-connected, so each of the $B\mathrm{Tor}$'s we have defined are principal $G'_g$-bundles over the corresponding $B\mathrm{Diff}$'s. In particular, their rational cohomologies are both $\bQ$-algebras and $G'_g$-representations, and we will describe them as such in a stable range. Before doing so, we recall that the work of Borel identifies
$$H^*(G'_g ; \bQ) = \begin{cases}
\bQ[\sigma_2, \sigma_6, \sigma_{10}, \ldots] & \text{ if $n$ is odd,}\\
\bQ[\sigma_4, \sigma_8, \sigma_{12}, \ldots] &  \text{ if $n$ is even.}
\end{cases} $$
in a stable range of degrees, where $\sigma_{4i-2n}$ may be chosen so that it pulls back to the Miller--Morita--Mumford class $\kappa_{\cL_i} \in H^{4i-2n}(B\Diff^+(W_g;\bQ)$ associated to the $i$th Hirzebruch $\mathcal{L}$-class. In particular the $\kappa_{\cL_i}$ vanish in the cohomology of $B\mathrm{Tor}^+(W_g)$.

Let us write $H(g) := H^n(W_g ; \bQ)$, which is the standard representation of $G'_g$. Pulled back from $B\Diff^+(W_g)$ to $B\mathrm{Tor}^+(W_g)$ the coefficient system $\cH$ is canonically trivialised, but has an action of $G'_g$: it can be identified with the dual $H(g)^\vee$. The edge homomorphism of the Serre spectral sequence
\begin{equation}\label{eq:edge}
H^*(B\Diff^+(W_g) ; \mathcal{H}^{\otimes S}) \lra \left[H^*(B\mathrm{Tor}^+(W_g);\bQ) \otimes (H(g)^\vee)^{\otimes S}\right]^{G'_g}
\end{equation}
allows us to consider the modified twisted Miller--Morita--Mumford classes $\bar{\kappa}_{\epsilon^S c}$ as providing $G'_g$-equivariant homomorphisms
$$\bar{\kappa}_c : H(g)^{\otimes S} \lra H^{n(|S|-2) + |c|}(B\mathrm{Tor}^+(W_g);\bQ).$$

The identities from the modified contraction formula correspond to identities among these maps: this will give relations analogous to \cite[Section 5.2]{KR-WTorelli}, which we will spell out after the proof of Theorem \ref{thm:TorelliCalc} below. First we explain how these relations can be organised in a categorical way, as follows.

Considering \eqref{eq:edge} as a natural transformation of functors on $\mathsf{(s)Br}_{2g}$, we may precompose it with the map
$$\bar{\kappa} : \mathcal{G}\mathrm{raph}_g(-) \lra H^*(B\Diff^+(W_g) ; \mathcal{H}^{\otimes -})$$
(which is an isomorphism in a stable range for $n \neq 2$ by Theorem \ref{thm:TwistedCoh}). This gives $G'_g$-equivariant maps $H(g)^{\otimes S} \otimes \mathcal{G}\mathrm{raph}_g(S) \to H^{*}(B\mathrm{Tor}^+(W_g);\bQ)$ which assemble to a map
$$K^\vee \otimes^{\mathsf{(s)Br}_{2g}} \mathcal{G}\mathrm{raph}_g(-) \lra H^{*}(B\mathrm{Tor}^+(W_g);\bQ)$$
out of the coend, where $K : \mathsf{(s)Br}_{2g} \to \mathsf{Rep}(G'_g)$ sends $S$ to $H(g)^{\otimes S}$. The domain obtains a graded-commutative $\bQ$-algebra structure coming from the lax symmetric monoidality of $\mathcal{G}\mathrm{raph}_g(-)$ and strong symmetric monoidality of $K(-)$. Theorem \ref{thm:TorelliCalc} below will say that this is surjective in a stable range, with kernel the ideal generated by the $\kappa_{\cL_i}$, but before stating it we explain a simplification.

Let us write $i : \mathsf{d(s)Br} \to \mathsf{(s)Br}_{2g}$ for the inclusion of the downward (signed) Brauer category. Thus subcategory is independent if $g$, as no circles can be created by composing morphisms in the downward Brauer category. Write $\mathcal{G}\mathrm{raph}_1(-)' \subset i^*\mathcal{G}\mathrm{raph}_1(-)_g$ for the subfunctor where we forbid bivalent vertices labelled by $1 \in \cV$ both of whose half-edges are legs; similarly, this functor is independent of $g$. Like just after \cite[Proposition 3.11]{KR-WTorelli}, $\mathcal{G}\mathrm{raph}_1(-)_g$ is then the left Kan extension $i_* \mathcal{G}\mathrm{raph}_1(-)'$ of $\mathcal{G}\mathrm{raph}_1(-)'$ along $i$.
We similarly define $\mathcal{G}\mathrm{raph}_*^\theta(-)'$, $\mathcal{G}\mathrm{raph}_*(-)'$, $\mathcal{G}\mathrm{raph}^\theta(-)'$, and $\mathcal{G}\mathrm{raph}(-)'$, whose left Kan extensions again recover the original functors. The following is the analogue of \cite[Theorem 4.1]{KR-WTorelli}.

\begin{theorem}\label{thm:TorelliCalc}
There are $G'_g$-equivariant ring homomorphisms
\begin{align*}
\frac{i^*(K^\vee) \otimes^{\mathsf{d(s)Br}} \mathcal{G}\mathrm{raph}_1(-)'}{(\kappa_{\cL_i} \, | \, 4i-2n > 0)}  &\lra H^*(B\mathrm{Tor}(W_g, D^{2n}); \bQ)\tag{i}\\
\frac{i^*(K^\vee) \otimes^{\mathsf{d(s)Br}} \mathcal{G}\mathrm{raph}^\theta_*(-)'}{(\kappa_{\cL_i} \, | \, 4i-2n > 0)}  &\lra H^*(B\mathrm{Tor}^\theta(W_g, *); \bQ)\tag{ii}\\
\frac{i^*(K^\vee) \otimes^{\mathsf{d(s)Br}} \mathcal{G}\mathrm{raph}_*(-)'}{(\kappa_{\cL_i} \, | \, 4i-2n > 0)}  &\lra H^*(B\mathrm{Tor}^+(W_g, *); \bQ)\tag{iii}\\
\frac{i^*(K^\vee) \otimes^{\mathsf{d(s)Br}} \mathcal{G}\mathrm{raph}^\theta(-)'}{({\kappa}_{\cL_i} \, | \, 4i-2n > 0)}  &\lra H^*(B\mathrm{Tor}^\theta(W_g); \bQ)\tag{iv}\\
\frac{i^*(K^\vee) \otimes^{\mathsf{d(s)Br}} \mathcal{G}\mathrm{raph}(-)'}{({\kappa}_{\cL_i} \, | \, 4i-2n > 0)}  &\lra H^*(B\mathrm{Tor}^{+}(W_g); \bQ)\tag{v}
\end{align*}
which for $2n \geq 6$ are isomorphisms in a stable range of degrees.

If $2n=2$ then, in a stable range of degrees and assuming that the target is finite-dimensional in degrees $* < N$ for all large enough $g$, these maps are isomorphisms onto the maximal algebraic subrepresentations in degrees $* \leq N$, and monomorphisms in degrees $* \leq N+1$.
\end{theorem}
\begin{proof}
By the main theorem of \cite{KR-WAlg}, as long as $2n \geq 6$ the $G'_g$-representations $H^i(B\mathrm{Tor}(W_g, D^{2n}); \bQ)$ are algebraic. Using the inheritance properties for algebraic representations from \cite[Theorem 2.2]{KR-WAlg}, the Serre spectral sequences for the homotopy fibre sequences
\begin{align*}
B\Tor(W_g, D^{2n}) \lra &B\Tor^{+}(W_g, *) \lra B\SO(2n)\\
B\Tor(W_g, D^{2n}) \lra &B\Tor^{\theta}(W_g, *) \lra B\SO(2n)\langle n \rangle
\end{align*}
show that the cohomology groups of $B\Tor^{+}(W_g, *)$ and $B\Tor^{\theta}(W_g, *)$ are also algebraic $G_g'$-representations, and the same for the homotopy fibre sequences
\begin{align*}
W_g \lra &B\Tor^{+}(W_g, *) \lra B\Tor^{+}(W_g)\\
W_g \lra &B\Tor^{\theta}(W_g, *) \lra B\Tor^{\theta}(W_g)
\end{align*}
show that the cohomology groups of $B\Tor^{+}(W_g)$ and $B\Tor^{\theta}(W_g)$ are algebraic $G_g'$-representations too.

Using this algebraicity property, case (i) is precisely \cite[Theorem 4.1]{KR-WTorelli}, using that by \cite[Proof of Theorem 5.1]{KR-WTorelli} $\mathcal{G}\mathrm{raph}_1(-)_g$ is isomorphic to the functor $\mathcal{P}(-, \cV)_{\geq 0} \otimes \det^{\otimes n}$. The other cases follow in the same way, using \cite[Proposition 2.16]{KR-WTorelli}, from Theorem \ref{thm:TwistedCoh}, with one elaboration which we describe below. The addendum in the case $2n=2$ is precisely as in \cite[Theorem 4.1]{KR-WTorelli}.

The elaboration comes when verifying the first hypothesis of \cite[Lemma 4.3]{KR-WTorelli}, which in case (v) for example requires us to know that $H^*(B\Diff^+(W_g) ; \mathcal{H}^{\otimes S})$ is a free $H^*(G'_g;\bQ)$-module in a stable range. But by transfer $H^*(B\Diff^+(W_g) ; \mathcal{H}^{\otimes S})$ is a summand of $H^*(B\Diff^+(W_g, *) ; \mathcal{H}^{\otimes S})$ (as $H^*(G'_g;\bQ)$-modules), and similarly with $\theta$-structures, so cases (ii) and (iii) imply cases (iv) and (v). In the other hand in case (iii) for example we have discussed in the proof of Theorem \ref{thm:TwistedCoh} the degeneration of the Serre spectral sequence in a stable range, giving
$$\mathrm{gr}(H^*(B\Diff^+(W_g, *) ; \mathcal{H}^{\otimes S})) \cong H^*(B\SO(2n);\bQ) \otimes H^*(B\Diff(W_g, D^{2n}) ; \mathcal{H}^{\otimes S}).$$
The Serre filtration is one of $H^*(G'_g;\bQ)$-modules, so as the associated graded is a free $H^*(G'_g;\bQ)$-module in a stable range (because $H^*(B\Diff(W_g, D^{2n}) ; \mathcal{H}^{\otimes S})$ is the case treated in \cite[Theorem 4.1]{KR-WTorelli}), it follows that $H^*(B\Diff^+(W_g, *) ; \mathcal{H}^{\otimes S})$ is too. The same argument applies in case (ii).
\end{proof}

This quite categorical description can be used to get a more down-to-earth presentation for these cohomology rings: in case (v) this is the presentation we have recorded in Theorem \ref{thm:A}. This is deduced just as in \cite[Section 5]{KR-WTorelli}, though a certain amount of the work can be avoided as we have already expressed things in terms of graphs rather than partitions. We outline the argument.

Let $R_\mathrm{pres}$ denote the $\bQ$-algebra whose presentation is given in Theorem \ref{thm:A} but \emph{omittting the relations (v) $\bar{\kappa}_{\cL_i}=0$}. By the universal property of coends, for any finite set $S$ there is a map $\bar{\kappa} : H(g)^{\otimes S} \otimes \mathcal{G}\mathrm{raph}(S)' \to i^*(K^\vee) \otimes^{\mathsf{d(s)Br}} \mathcal{G}\mathrm{raph}(-)'$, and so in particular there are maps
$$\bar{\kappa}_c : H(g)^{\otimes r} \lra i^*(K^\vee) \otimes^{\mathsf{d(s)Br}} \mathcal{G}\mathrm{raph}(-)'$$
corresponding to the graph given by the corolla with legs $\{1,2,\ldots,r\}$ and label $c \in \cW$. The identifications made in a coend mean that if $\omega = \sum_{i=1}^{2g} a_i \otimes a_i^\# \in H(g)^{\otimes 2}$ is the form dual to the intersection pairing then $\sum_i \bar{\kappa}_c(v_1 \otimes \cdots \otimes v_{r-2} \otimes a_i \otimes a_i^\#)$ is identified with
$$\bar{\kappa}(v_1 \otimes \cdots \otimes v_{r-2} \otimes \{\substack{\text{the graph obtained from the corolla with legs $\{1,2,\ldots,r\}$ and label} \\ \text{$c \in \cW$ by gluing an edge between its last two legs}}\}).$$
Using the modified contraction formula (relation (d$^{\prime\prime\prime}$)), and permuting the ordering of some half-edges, we have the identity 

\begin{figure}[h]
	\begin{tikzpicture}[arrowmark/.style 2 args={decoration={markings,mark=at position #1 with \arrow{#2}}}]
\begin{scope}[scale=0.5]

	\draw[dashed] (0,0) circle (1.8cm);
	\draw [thick,NavyBlue]  (0,.3) circle (0.8cm);
	\node at (0,-0.5) [NavyBlue] {$\bullet$};
	\draw [thick,NavyBlue] (0,-1.8) -- (0,-0.5);
	\draw [thick,NavyBlue] (-1,-1.5) -- (0,-0.5);
	\draw [thick,NavyBlue] (1,-1.5) -- (0,-0.5);
	\node at (0,-.5) [above] {$c$};
	
	\draw [NavyBlue] (0.3,-1.2) -- (0.7,-0.8);
	\node at (-1.5,-1.5) [below] {$r-2$};
	\node at (0,-1.8) [below] {$\cdots$};
	\node at (1,-1.5) [below] {$1$};

	\node at (2.5,0) {=};
	
	\begin{scope}[xshift=7cm]
	\node at (-3,0) {$\tfrac{\chi-2}{\chi}$};
	\draw[dashed] (0,0) circle (1.8cm);
	\node at (0,-0.5) [NavyBlue] {$\bullet$};
%	\draw [NavyBlue] (-0.3,-1.1) -- (-0.7,-0.7);
	\draw [thick,NavyBlue] (0,-1.8) -- (0,-0.5);
	\draw [thick,NavyBlue] (-1,-1.5) -- (0,-0.5);
		\draw [thick,NavyBlue] (1,-1.5) -- (0,-0.5);
	\node at (0,-.5) [above] {$ce$};
	
	\draw [NavyBlue] (0.3,-1.2) -- (0.7,-0.8);
	\node at (-1.5,-1.5) [below] {$r-2$};
	\node at (0,-1.8) [below] {$\cdots$};
	\node at (1,-1.5) [below] {$1$};
	
	\end{scope}
	
	\begin{scope}[xshift=13cm]
	\node at (-3,0) {$+\tfrac{1}{\chi^2}$};
	\draw[dashed] (0,0) circle (1.8cm);
	\node at (0,-0.5) [NavyBlue] {$\bullet$};
%	\draw [NavyBlue] (-0.3,-1.1) -- (-0.7,-0.7);
	\draw [thick,NavyBlue] (0,-1.8) -- (0,-0.5);
	\draw [thick,NavyBlue] (-1,-1.5) -- (0,-0.5);
		\draw [thick,NavyBlue] (1,-1.5) -- (0,-0.5);
	\node at (0,-.5) [above] {$c$};
	\node at (1,0) [NavyBlue] {$\bullet$};
	\node at (1,0) [above] {$e^2$};
	
	\draw [NavyBlue] (0.3,-1.2) -- (0.7,-0.8);
	\node at (-1.5,-1.5) [below] {$r-2$};
	\node at (0,-1.8) [below] {$\cdots$};
	\node at (1,-1.5) [below] {$1$};
	
	\end{scope}

	\end{scope}
	\end{tikzpicture}
\end{figure}
\noindent in $\mathcal{G}\mathrm{raph}(\{1,2,\ldots,r-2\})'$, showing that $\sum_i \bar{\kappa}_c(v_1 \otimes \cdots \otimes v_{r-2} \otimes a_i \otimes a_i^\#) = \tfrac{\chi-2}{\chi} \bar{\kappa}_{ce}(v_1 \otimes \cdots \otimes v_{r-2}) + \tfrac{1}{\chi^2} \bar{\kappa}_{e^2} \bar{\kappa}_c(v_1 \otimes \cdots \otimes v_{r-2})$, i.e.\ that relation (iv) in the presentation of $R_\mathrm{pres}$ holds in the algebra $i^*(K^\vee) \otimes^{\mathsf{d(s)Br}} \mathcal{G}\mathrm{raph}(-)'$. Similarly one shows that relations (i)--(iii) in the presentation of $R_\mathrm{pres}$ hold, so that there is defined an algebra map
$$\phi : R_\mathrm{pres} \lra i^*(K^\vee) \otimes^{\mathsf{d(s)Br}} \mathcal{G}\mathrm{raph}(-)'.$$
This is in addition a map of algebraic $G'_g$-representations, so to show it is an isomorphism in a stable range it suffices to show that it induces an isomorphism on applying $[- \otimes H(g)^{\otimes S}]^{G'_g}$ in a stable range, for all finite sets $S$. This is done, as in the proof of \cite[Theorem 5.1]{KR-WTorelli}, by constructing a factorisation
$$\mathcal{G}\mathrm{raph}(S) \overset{\alpha}\lra [R_\mathrm{pres} \otimes H(g)^{\otimes S}]^{G'_g} \overset{\phi_*}\lra [i^*(K^\vee) \otimes^{\mathsf{d(s)Br}} \mathcal{G}\mathrm{raph}(-)' \otimes H(g)^{\otimes S}]^{G'_g}$$ 
of the canonical map, which is an isomorphism in a stable range. The map $\alpha$ is surjective by construction, so is an isomorphism in a stable range, and hence $\phi_*$ is too. Thus $\phi$ is an isomorphism in a stable range: quotienting out both sides by the ideal genearted by the $\bar{\kappa}_{\cL_i}$ gives Theorem \ref{thm:A}. This finishes the outline of the argument.

\section{The case $2n=2$}\label{sec:surfaces} 

Although Theorem \ref{thm:TorelliCalc} is only known to hold in a limited range of degrees in the case $2n=2$ ($N=2$ is currently the best known constant for $g \geq 3$, using the work of Johnson \cite{JohnsonIII}), Theorem \ref{thm:TwistedCoh} does hold in a range of cohomological degrees tending to infinity with $g$. In this case our discussion is closely related to the work of Kawazumi and Morita \cite{MoritaLinearRep, KM, KMunpub}, and in this section we take the opportunity to revisit that work from our perspective. \emph{Throughout this section we assume that $g \geq 2$}, so that $\chi(W_g) = 2-2g \neq 0$.

In terms of Kawazumi and Morita's notation we have
$$\mathcal{M}_g := \pi_0(\mathrm{Diff}^+(W_g)) \quad \mathcal{M}_{g,*} := \pi_0(\mathrm{Diff}^+(W_g, *)) \quad \mathcal{M}_{g,1} := \pi_0(\mathrm{Diff}^+(W_g, D^2)).$$
Under our assumption $g \geq 2$ the groups $\mathrm{Diff}^+(W_g)$, $\mathrm{Diff}^+(W_g, *)$, and $\mathrm{Diff}^+(W_g, D^2)$ all have contractible path-components, so the group cohomology of $\mathcal{M}_g$ is the cohomology of $B\mathrm{Diff}^+(W_g)$, and so on. Theorem \ref{thm:TwistedCoh} gives a natural transformation
$$\bar{\kappa} : \mathcal{G}\mathrm{raph}(-)_g \lra H^*(\mathcal{M}_g ; \cH^{\otimes -})$$
of functors $\mathsf{sBr}_{2g} \to \mathsf{Gr}(\bQ\text{-}\mathsf{mod})$, which is an isomorphism in a stable range of degrees. Note that in this case $H^*(B\SO(2);\bQ)=\bQ[e]$ so $\mathcal{V} = \mathcal{W} = \bQ[e]$ and there is no difference between the tangential structure $\theta$ and an orientation. In particular if we denote by $\Gamma_i \in \mathcal{G}\mathrm{raph}(\emptyset)$ the graph with a single vertex, no edges, and labelled by $e^{i+1}$, then $\bar{\kappa}(\Gamma_i) = \kappa_i \in H^{2i}(\mathcal{M}_g;\bQ)$ is the usual Miller--Morita--Mumford class\footnote{Our $\kappa_i$ is denoted $e_i$ in the work of Kawazumi and Morita.}.

Our goal in this Section is to relate our $\mathcal{G}\mathrm{raph}(-)$ to the work of Kawazumi and Morita mentioned above as well as work of Akazawa \cite{Akazawa} and Garoufalidis and Nakamura \cite{GN, GNCorr}. A brief overview is as follows. The modified contraction formula can be used to simplify elements of $\mathcal{G}\mathrm{raph}(S)$ in two different way. Firstly, in Section \ref{sec:RedCor} the modified contraction formula is used to express a graph as a linear combination of graphs with fewer edges, and hence by induction to express it as a linear combination of graphs with no edges. In particular this shows that $\mathcal{G}\mathrm{raph}(\emptyset)_g \cong \bQ[\Gamma_1, \Gamma_2, \ldots]$ for the graphs $\Gamma_i$ described above. Secondly, in Section \ref{sec:RedTriv} the modified contraction formula is used ``in reverse'' to express a graph as a linear combination of graphs with smaller powers of $e$ in a label, or with smaller valence, and hence by induction to express it as a linear combination of trivalent graphs with all labels 1. Reduced to this setting of trivalent graphs, in Section \ref{sec:orderings} it is explained how the decorations by ordering and orientation data can be neglected, due to the existence of a normal form for it. This allows for all elements of $\mathcal{G}\mathrm{raph}(S)$ to be expressed in terms of undecorated trivalent graphs, and  Section \ref{sec:TrivRels} describes the complete set of relations among undecorated trivalent graphs which need to be imposed in order to present $\mathcal{G}\mathrm{raph}(S)$: we rediscover a modified form of the ``$I = H$'' relation among trivalent graphs, which has arisen in the work of Akazawa \cite{Akazawa} and Garoufalidis and Nakamura \cite{GNCorr}. Finally, in Section \ref{sec:GN} we explain how this discussion completes a calculation in symplectic invariant theory which has been outstanding since \cite{GNCorr}.

\subsection{Reduction to corollas}\label{sec:RedCor}

  The possible labels for the vertices of graphs in $\mathcal{G}\mathrm{raph}(S)$ are powers of the Euler class $e$. Given any graph we may iteratedly apply the modified contraction formula to write it as a linear combination of graphs with fewer edges, and hence any graph is equivalent to a linear combination of graphs with no edges: these are disjoint unions of corollas. Of these, by definition of $\mathcal{G}\mathrm{raph}$: the 0-valent corolla labelled by $e$ is equal to the scalar $\chi$, the 1-valent corolla labelled by $1 \in \cV$ is trivial, and the 1-valent corolla labelled by $e \in \cV$ is trivial. Define a \emph{labelled partition} of a finite set $S$ to be a partition $\{S_\alpha\}_{\alpha \in I}$ of $S$ into (possibly empty) subsets and a label $e^{n_\alpha}$ for each part, such that
\begin{enumerate}[(i)]
\item If $|S_\alpha|=0$ then $n_\alpha \geq 2$,

\item If  $|S_\alpha|=1$ then $n_\alpha \geq 1$.
\end{enumerate}
We give a part $(S_\alpha, n_\alpha)$ degree $2n_\alpha + |S_\alpha|-2$, and a labelled partition the degree given by the sums of the degrees of its parts. Similarly to the proof of Theorem \ref{thm:TwistedCoh} (iv) (particularly around equation \eqref{eq:MapToGraph}), let $\mathcal{P}^\text{bis}(S, \cV)_{\geq 0}$ denote the free $\bQ[\chi^{\pm 1}]$-module with basis the set of labelled partitions of $S$. Assigning to a labelled part $(S_\alpha, e^{n_\alpha})$ the corolla with legs $S_\alpha$ and label $e^{n_\alpha}$ defines a map
\begin{equation}\label{eq:PbisAgain}
\mathcal{P}^\text{bis}(S,\cV)_{\geq 0} \otimes \det \bQ^S \lra \mathcal{G}\mathrm{raph}(S),
\end{equation}
natural in $S$ with respect to bijections. 

\begin{lemma}\label{lem:JustGammas}
The map \eqref{eq:PbisAgain} is an isomorphism.
\end{lemma}
\begin{proof}
It is surjective, as explained above, by repeatedly applying the modified contraction formula to express a graph in terms of graphs without edges.

If it were not injective then it would have some nontrivial $\bQ[\chi^{\pm 1}]$-linear combination of labelled partitions in its kernel, of a given degree $d$, and this would remain a nontrivial $\bQ$-linear combination of labelled partitions when specialised to $\chi = 2-2g$ for all $g \gg 0$ (as a Laurent polynomial in $\chi$ has finitely-many roots). But in the proof of Theorem \ref{thm:TwistedCoh} (iv), in the discussion after equation \eqref{eq:MapToGraph}, it is explained that when specialised to $\chi = 2-2g$ this map is an isomorphism in a range of degrees tending to infinity with $g$; for large enough $g$ the degree $d$ will be in this stable range, a contradiction.
\end{proof}

In particular, using the graphs $\Gamma_i$ described above there is an isomorphism
\begin{equation}\label{eq:CalcGraphsEmpty}
\bQ[\chi^{\pm 1}][\Gamma_1, \Gamma_2, \ldots] \cong \mathcal{G}\mathrm{raph}(\emptyset).
\end{equation}

\subsection{Reduction to trivalent graphs without labels.}\label{sec:RedTriv}

In this section we will prove the following.

\begin{theorem}\label{thm:AllAreTriv}
Using the modified contraction formula any marked oriented graph is equivalent to a $\bQ[\chi^{\pm1}, (\chi-2)^{-1}, (\chi-3)^{-1}, (\chi-4)^{-1}]$-linear combination of trivalent graphs with all vertices labelled by $1 \in \mathcal{V}_0$.
\end{theorem}

Let $\mathcal{G}\mathrm{raph}^\text{tri}(S) \leq \mathcal{G}\mathrm{raph}(S)$ denote the sub-$\bQ[\chi^{\pm1}]$-module spanned by those marked oriented graphs which are trivalent and all of whose labels are $1\in \mathcal{V}$.

\begin{corollary}\label{cor:TriIsAll}
The monomorphism $i : \mathcal{G}\mathrm{raph}^\text{tri}(-) \to \mathcal{G}\mathrm{raph}(-)$ becomes an isomorphism upon inverting $\chi-2$, $\chi-3$, and $\chi-4$. In particular $\mathcal{G}\mathrm{raph}^\text{tri}(-)_g = \mathcal{G}\mathrm{raph}(-)_g$.
\end{corollary}

\begin{remark}[2-valent vertices labelled by 1]\label{rem:RemoveBivalent}
Using the relation
$$\lambda_{2,3}(\kappa_{\bar{\epsilon}^{1,2}} \kappa_{\bar{\epsilon}^{3,\ldots, n} c}) = \kappa_{\bar{\epsilon}^{1,3,\ldots, n} c} $$
we can always remove 2-valent vertices labelled by 1. It is sometimes convenient when writing formulas for 3-valent graphs to also allow 2-valent vertices labelled by 1: we allow ourselves to do so, noting that the above can always be used to eliminate the 2-valent vertices.
\end{remark}

\begin{proof}[Proof of Theorem \ref{thm:AllAreTriv}]
As a matter of notation we will formally manipulate modified twisted Miller--Morita--Mumford classes, but this is equivalent to manipulating marked oriented graphs. Rearranging the first contraction formula gives
\begin{equation}\label{eq:Rel1}
\kappa_{\bar{\epsilon}^a e^b} = \tfrac{\chi}{\chi-2}\left( \lambda_{1,2}\kappa_{\bar{\epsilon}^{2+a} e^{b-1}} -\tfrac{1}{\chi^2} \kappa_{e^2}  \kappa_{\bar{\epsilon}^a e^{b-1}}\right).
\end{equation}
Rearranging the second contraction formula gives
$$\kappa_{\bar{\epsilon}^{a+b}} = \lambda_{a+1, a+2}(\kappa_{\bar{\epsilon}^{a+1}} \cdot \kappa_{\bar{\epsilon}^{1+b}}) - \tfrac{1}{\chi^2}(\kappa_{e^2} \cdot \kappa_{\bar{\epsilon}^{a}} \cdot \kappa_{\bar{\epsilon}^{b}}) + \tfrac{1}{\chi}(\kappa_{\bar{\epsilon}^{a} e} \cdot \kappa_{\bar{\epsilon}^{b}} + \kappa_{\bar{\epsilon}^{a}} \cdot \kappa_{\bar{\epsilon}^{b} e})$$
and using \eqref{eq:Rel1} to eliminate the Euler classes from the last two terms gives
\begin{align*}
\kappa_{\bar{\epsilon}^{a+b}} &= \lambda_{a+1, a+2}(\kappa_{\bar{\epsilon}^{a+1}} \cdot \kappa_{\bar{\epsilon}^{1+b}}) - \tfrac{1}{\chi^2}(\kappa_{e^2} \cdot \kappa_{\bar{\epsilon}^{a}} \cdot \kappa_{\bar{\epsilon}^{b}})\\
&\quad + \tfrac{1}{\chi-2}((\lambda_{1,2}(\kappa_{\bar{\epsilon}^{2+a}}) - \tfrac{1}{\chi^2}\kappa_{e^2} \kappa_{\bar{\epsilon}^a}) \cdot \kappa_{\bar{\epsilon}^{b}} + \kappa_{\bar{\epsilon}^{a}} \cdot (\lambda_{1,2}(\kappa_{\bar{\epsilon}^{2+b}}) - \tfrac{1}{\chi^2}\kappa_{e^2} \kappa_{\bar{\epsilon}^b}))\\
&= \lambda_{a+1, a+2}(\kappa_{\bar{\epsilon}^{a+1}} \cdot \kappa_{\bar{\epsilon}^{1+b}}) + \tfrac{1}{\chi-2}\left((\lambda_{1,2}(\kappa_{\bar{\epsilon}^{2+a}})\cdot \kappa_{\bar{\epsilon}^{b}} + \kappa_{\bar{\epsilon}^{a}} \cdot \lambda_{1,2}(\kappa_{\bar{\epsilon}^{2+b}})\right)\\
&\quad - \tfrac{1}{\chi(\chi-2)} \kappa_{e^2} \cdot \kappa_{\bar{\epsilon}^a} \cdot \kappa_{\bar{\epsilon}^{b}}.
\end{align*}

It suffices to show that each corolla $\kappa_{\bar{\epsilon}^a e^b}$ may be represented by a linear combination of trivalent graphs. By Example \ref{ex:theta} the class $\kappa_{e^2}$ may be represented by a trivalent graph (after inverting $\chi-3$) so by iteratedly applying \eqref{eq:Rel1} it suffices to show that each $\kappa_{\bar{\epsilon}^n}$ can too. By Remark \ref{rem:RemoveBivalent} we may as well show that classes can be represented by 2- and 3-valent graphs. To get started we have $\kappa_{\bar{\epsilon}}=0$ as it has negative degree.

Consider the class $\lambda_{2,5} \lambda_{3,4}(\kappa_{\bar{\epsilon}^{1,2,3}} \cdot \kappa_{\bar{\epsilon}^{4,5,6}})$. Using the form of the relations above, which avoid creating Euler classes, this is
\begin{align*}
&\lambda_{2,5}(\kappa_{\bar{\epsilon}^{1,2,5,6}} - \tfrac{1}{\chi-2}(\lambda_{u,v}(\kappa_{\bar{\epsilon}^{u,v,1,2}})\kappa_{\bar{\epsilon}^{5,6}}+\kappa_{\bar{\epsilon}^{1,2}}\lambda_{u,v}(\kappa_{\bar{\epsilon}^{u,v,5,6}})) + \tfrac{1}{\chi(\chi-2)}(\kappa_{e^2}\kappa_{\bar{\epsilon}^{1,2}}\kappa_{\bar{\epsilon}^{5,6}}))\\
&=\lambda_{2,5}(\kappa_{\bar{\epsilon}^{1,2,5,6}}) - \tfrac{2}{\chi-2}\lambda_{u,v}(\kappa_{\bar{\epsilon}^{u,v,1,6}})  + \tfrac{1}{\chi(\chi-2)} \kappa_{e^2}\kappa_{\bar{\epsilon}^{1,6}}\\
&= \tfrac{\chi-4}{\chi-2}\lambda_{2,5}(\kappa_{\bar{\epsilon}^{1,2,5,6}}) + \tfrac{1}{\chi(\chi-2)} \kappa_{e^2}\kappa_{\bar{\epsilon}^{1,6}}
\end{align*}
Renumbering legs and rearranging, this shows that $\lambda_{1,2}(\kappa_{\bar{\epsilon}^4})$ may be represented by 2- and 3-valent graphs.

Applied with $(a,b)=(2,2)$ the second relation gives
$$\kappa_{\bar{\epsilon}^{4}} = \lambda_{3, 4}(\kappa_{\bar{\epsilon}^{3}} \cdot \kappa_{\bar{\epsilon}^{3}}) + \tfrac{1}{\chi-2}\left((\lambda_{1,2}(\kappa_{\bar{\epsilon}^{4}})\cdot \kappa_{\bar{\epsilon}^{2}} + \kappa_{\bar{\epsilon}^{2}} \cdot \lambda_{1,2}(\kappa_{\bar{\epsilon}^{4}})\right) - \tfrac{1}{\chi(\chi-2)} \kappa_{e^2} \cdot \kappa_{\bar{\epsilon}^2} \cdot \kappa_{\bar{\epsilon}^{2}},$$
which with the above shows that $\kappa_{\bar{\epsilon}^{4}}$ may be represented by 2- and 3-valent graphs.

Similarly to the above, consider $\lambda_{2,5} \lambda_{3,4}(\kappa_{\bar{\epsilon}^{1,2,3}} \cdot \kappa_{\bar{\epsilon}^{4,5,6,7}})$, which is
\begin{align*}
&\lambda_{2,5}(\kappa_{\bar{\epsilon}^{1,2,5,6,7}} + \tfrac{1}{\chi(\chi-2)} \kappa_{e^2} \kappa_{\bar{\epsilon}^{1,2}}\kappa_{\bar{\epsilon}^{5,6,7}} - \tfrac{1}{\chi-2}(\lambda_{u,v}(\kappa_{\bar{\epsilon}^{u,v,1,2}})\kappa_{\bar{\epsilon}^{5,6,7}} + \kappa_{\bar{\epsilon}^{1,2}}\lambda_{u,v}(\kappa_{\bar{\epsilon}^{u,v,5,6,7}})))\\
&= \lambda_{2,5}(\kappa_{\bar{\epsilon}^{1,2,5,6,7}}) + \tfrac{1}{\chi(\chi-2)} \kappa_{e^2} \kappa_{\bar{\epsilon}^{1,6,7}} - \tfrac{1}{\chi-2}(\lambda_{2,5}\lambda_{u,v}(\kappa_{\bar{\epsilon}^{u,v,1,2}}\kappa_{\bar{\epsilon}^{5,6,7}}) + \lambda_{u,v}(\kappa_{\bar{\epsilon}^{u,v,1,6,7}}))\\
&= \tfrac{\chi-3}{\chi-2}\lambda_{2,5}(\kappa_{\bar{\epsilon}^{1,2,5,6,7}}) + \tfrac{1}{\chi(\chi-2)} \kappa_{e^2} \kappa_{\bar{\epsilon}^{1,6,7}} - \tfrac{1}{\chi-2}\lambda_{2,5}\lambda_{u,v}(\kappa_{\bar{\epsilon}^{u,v,1,2}})\kappa_{\bar{\epsilon}^{5,6,7}}.
\end{align*}
Renumbering legs and rearranging, this shows that $\lambda_{1,2}(\kappa_{\bar{\epsilon}^5})$ may be represented by 2-, 3-, and 4-valent graphs; with the above it follows that it can also be represented by 2- and 3-valent graphs.

Applied with $(a,b)=(2,3)$ the second relation gives
$$\kappa_{\bar{\epsilon}^{5}} = \lambda_{3, 4}(\kappa_{\bar{\epsilon}^{3}} \cdot \kappa_{\bar{\epsilon}^{4}}) + \tfrac{1}{\chi-2}\left((\lambda_{1,2}(\kappa_{\bar{\epsilon}^{4}})\cdot \kappa_{\bar{\epsilon}^{3}} + \kappa_{\bar{\epsilon}^{2}} \cdot \lambda_{1,2}(\kappa_{\bar{\epsilon}^{5}})\right) - \tfrac{1}{\chi(\chi-2)} \kappa_{e^2} \cdot \kappa_{\bar{\epsilon}^2} \cdot \kappa_{\bar{\epsilon}^{3}},$$
so it follows that $\kappa_{\bar{\epsilon}^{5}}$ may be represented by 2- and 3-valent graphs.

If $n \geq 6$ then we can write $n=a+b$ with $a, b \geq 3$, so $a+2, b+2 < n$ and so the second relation expresses $\kappa_{\bar{\epsilon}^{n}}$ in terms of $\kappa_{\bar{\epsilon}^{m}}$'s with $m < n$. Thus all $\kappa_{\bar{\epsilon}^{n}}$'s may be represented by 2- and 3-valent graphs as required.
\end{proof}

It is worth observing that we have the relation 
\begin{equation}\label{eq:NoLoop}
\lambda_{1,2}(\kappa_{\bar{\epsilon}^3}) = \tfrac{\chi-2}{\chi} \kappa_{\bar{\epsilon} e} + \tfrac{1}{\chi^2} \kappa_{e^2} \kappa_{\bar{\epsilon}}=0,
\end{equation}
using that $\kappa_{\bar{\epsilon} e}=0$ (by definition) and that $\kappa_{\bar{\epsilon}}=0$ (as it has negative degree). This means that any graph having a trivalent vertex with a loop is trivial in $\mathcal{G}\mathrm{raph}(-)$.

\subsection{A remark on orderings.}\label{sec:orderings}

A curious normalisation is possible when considering trivalent graphs, allowing one to neglect the orderings of vertices, of half-edges, and the orientations of edges. In \cite{MoritaLinearRep, KM, KMunpub} this is implemented \emph{ab initio} and (marked) oriented graphs play no role. Let us explain this normalisation, extended to trivalent graphs with legs.

A trivalent graph $\tilde{\Gamma}$ with legs $S$ consists of a set $V$ of vertices, a set $H$ of half-edges, a 3-to-1 map $a : H \to V$ recording to which vertex each half-edge is incident, and an unordered matching $\mu$ on $H \sqcup S$ recording which half-edges span an edge, and which half-edges are connected to which legs. 

Given a trivalent graph $\tilde{\Gamma} = (V, H, a : H \to V, \mu)$ with legs $S$, we may choose an ordering of $V$ and choose an ordering of $H$ such that $a : H \to V$ is weakly monotone (equivalently, choose an ordering of the half-edges incident at each vertex). We also choose an ordering of $S$. There is an induced ordering of $H \sqcup S$ by putting $\vec{S}$ after $\vec{H}$, and we form an ordered matching $m$ of $H \sqcup S$ by taking those pairs $(a,b)$ with $a < b$ and $\{a,b\} \in \mu$. Using this we form an oriented trivalent graph $\Gamma_\text{choice} = (\vec{V}, \vec{H}, a : H \to V, m)$, depending on these choices. The normalisation is as follows. Let $x_1 < x_2 < x_3 < x_4 < \ldots < x_{2k} \in H \sqcup S$ be the total order on $H \sqcup S$, and let $a_1 < b_1, \ldots, a_k < b_k$ be the ordered pairs which span an edge, with $a_1 < a_2 < \ldots < a_k \in H \sqcup S$. Then there is a bijection given by
$$\rho := \bigl(\begin{smallmatrix}
    a_1 & b_1 & a_2 & b_2 & a_3 & b_3 & \cdots  & a_k & b_k \\
    x_1 & x_2 & x_3 & x_4 & x_5 & x_6 & \cdots  & x_{2k-1} & x_{2k}
  \end{smallmatrix}\bigr)$$
and we define $\Gamma := \sign(\rho) \cdot \Gamma_\text{choice}$.
\vspace{2ex}

\noindent\textbf{Claim.} As long as $\tilde{\Gamma}$ has no vertices with loops, the element $\Gamma$ does not depend on the choice of ordering of $V$ or $H$, and depends on the ordering of $S$ precisely as the sign representation. 

\vspace{1ex}

In particular if we set\footnote{In \cite{MoritaLinearRep, KM, KMunpub} they restrict to ``trivalent graphs without loops", however we find it more natural to allow loops but set graphs with a loop to zero.}
$$\mathcal{G}\mathrm{raph}^\text{undec}(S) := \bQ[\chi^{\pm 1}][\tilde{\Gamma} \text{ trivalent graph with legs $S$}]/(\text{graphs with loops})$$
then the Claim together with the relation \eqref{eq:NoLoop} provides an epimorphism 
$$\Phi : \mathcal{G}\mathrm{raph}^\text{undec}(S) \otimes \det \bQ^S \lra \mathcal{G}\mathrm{raph}^\text{tri}(S)$$
of $\bQ[\chi^{\pm 1}]$-modules, natural with respect to bijections in $S$. This can be extended to a natural transformation of functors on $\mathsf{sBr}^\chi$ by letting an ordered matching $(a,b)$ of elements of $S$ act by adding an edge to the trivalent graph connecting $a$ and $b$, and contracting the determinant by $a \wedge b$. Doing so might create a circle with no vertices, which should be replaced by the scalar $\chi-2$.

\begin{proof}[Proof of Claim]
If $(h_1, h_2, h_3)$ are the half-edges incident at a vertex $v$ and we change their ordering to $(h_{\sigma(1)}, h_{\sigma(2)}, h_{\sigma(3)})$ giving $\Gamma_\text{choice}'$, then (under the assumption that $\Gamma$ does not have loops) the relative ordering of half-edges forming an edge has not changed, so $m'=m$. Thus $\Gamma_\text{choice}' = \sign(\sigma) \cdot \Gamma_\text{choice}$. On the other hand $\rho'$ is obtained from $\rho$ by postcomposing with $\sigma$, and precomposing with a permutation which permutes some $(a_i < b_i)$'s, which is an even permutation. Thus $\sign(\rho') = \sign(\sigma) \cdot \sign(\rho)$, so $\Gamma' = \Gamma$.

Suppose a vertex $v^1$ has half edges $(h_1^1, h_2^1, h_3^1)$ and $v^2$ has half edges $(h_1^2, h_2^2, h_3^2)$, and $v^1 < v^2 \in \vec{V}$ are adjacent in the ordering on $V$, and consider transposing the ordering of these vertices. For edges between a $u<v^1$ and a $v^i$ or between a $v^i$ and a $u > v^2$ the relative ordering of their half-edges does not change. Edges between $v^1$ and $v^2$ have the relative ordering of their half-edges reversed. Thus if there are $N$ such edges we have $\Gamma_\text{choice}' = (-1)^{1+N} \cdot\Gamma_\text{choice}$. But the permutation $\rho$ is changed by permuting $(h_1^2, h_2^2, h_3^2)$ past $(h_1^1, h_2^1, h_3^1)$, which has sign $-1$, and $N$ transpositions $(a_i b_i)$, which has sign $(-1)^N$. Thus again $\Gamma' = \Gamma$.

Finally, changing the order of $S$ by a permutation $\tau$ changes $\rho$ by postcomposition with $\tau$, so acts as $\sign(\tau)$.
\end{proof}

\begin{example}\label{ex:ThetaUndec}
For the ordering of vertices and half-edges corresponding to the theta-graph in Example \ref{ex:theta} the associated permutation is $\rho=(1)(2 3 5)(4 6)$ which is odd, so the undecorated theta-graph yields $\tfrac{\chi-3}{\chi}\kappa_{e^2}$. This is precisely minus the evaluation of $\beta_{\Gamma_2}$ on \cite[p.\ 39]{KMunpub} (unfortunately the theta-graph is denoted $\Gamma_2$ in that paper). This minus comes from the use of a different sign convention, see the discussion at \cite[top of p.\ 33]{KR-WTorelli}.
\end{example}

\subsection{Relations among trivalent graphs.}\label{sec:TrivRels}

The modified contraction formula describes relations among graphs involving contracting an edge, but this necessarily involves graphs with vertices of different valencies. In Theorem \ref{thm:AllAreTriv} we have explained that, in the case of surfaces, all graphs may be expressed purely in terms of trivalent graphs: one may ask what relations among trivalent graphs $\Gamma$ are imposed by the contraction formula.

For the unmodified contraction formula discussed in \cite{KR-WTorelli}, the answer is that it imposes the ``$I=H$" relation among trivalent graphs: this is because both the $I$- and $H$-graphs admit contractions to the $X$-graph. Furthermore, as all connected trivalent graphs with the same number of legs and of the same genus are equivalent under the ``$I=H$" relation, and the contraction formula never changes the genus or number of legs, there are no further relations.

In the setting of the modified contraction formula discussed here it is more complicated. It is best given in the setting of undecorated trivalent graphs.

\begin{theorem}\label{thm:AkazawaRel}
After inverting $\chi-2$, $\chi-3$, and $\chi-4$, undecorated trivalent graphs which differ locally by
\begin{equation}
	\begin{tikzpicture}[baseline=(current bounding box.center), arrowmark/.style 2 args={decoration={markings,mark=at position #1 with \arrow{#2}}}]
	\begin{scope}[scale=0.5]
	
	\draw[dashed] (1,0) circle (1.414cm);
	\draw [thick,NavyBlue] (0,0) -- (2,0);
	\draw [thick,NavyBlue] (0,-1) -- (0,1);
	\draw [thick,NavyBlue] (2,-1) -- (2,1);

	\node at (3,0) {$=$};
	
		\draw[dashed] (5,0) circle (1.414cm);
		\draw [thick,NavyBlue] (4,1) -- (6,1);
		\draw [thick,NavyBlue] (5,-1) -- (5,1);
		\draw [thick,NavyBlue] (4,-1) -- (6,-1);
		
		\node at (9,0) {$+\,\, \frac{1}{(\chi-4)(3-\chi)}  ($};
		
		\draw[dashed] (13,0) circle (1.414cm);
		\draw [thick,NavyBlue] (12,-1) -- (12,1);
		\draw [thick,NavyBlue] (14,-1) -- (14,1);

		\draw [thick,NavyBlue] (12.5,0) -- (13.5,0);		
		\draw[thick, NavyBlue](12.5,0)to[bend left=90](13.5,0);
		\draw[thick, NavyBlue](12.5,0)to[bend right=90](13.5,0);

		\node at (15,0) {$-$};
		
		\draw[dashed] (17,0) circle (1.414cm);
		\draw [thick,NavyBlue] (16,1) -- (18,1);
		\draw [thick,NavyBlue] (16,-1) -- (18,-1);
		
		\draw [thick,NavyBlue] (16.5,0) -- (17.5,0);		
		\draw [thick, NavyBlue](16.5,0)to[bend left=90](17.5,0);
		\draw [thick, NavyBlue](16.5,0)to[bend right=90](17.5,0);
		
		\node at (19,0) {$)$};
			
		\node at (0.3,-3) {$+ \tfrac{1}{\chi-4} ($};	
					
		\draw[dashed] (3,-3) circle (1.414cm);
		\draw [thick,NavyBlue] (2,-2) -- (2,-2.5);
		\draw [thick, NavyBlue](2,-2.5)to[bend right=60](2,-3.5);
		\draw [thick, NavyBlue](2,-2.5)to[bend left=60](2,-3.5);
		\draw [thick,NavyBlue] (2,-3.5) -- (2,-4);

		\draw [thick,NavyBlue] (4,-2) -- (4,-4);
		
		\node at (5,-3) {$+$};	
			
		\draw[dashed] (7,-3) circle (1.414cm);
		\draw [thick,NavyBlue] (6,-2) -- (6,-4);
		
		\draw [thick,NavyBlue] (8,-2) -- (8,-2.5);
		\draw [thick, NavyBlue](8,-2.5)to[bend right=60](8,-3.5);
		\draw [thick, NavyBlue](8,-2.5)to[bend left=60](8,-3.5);
		\draw [thick,NavyBlue] (8,-3.5) -- (8,-4);
		
		\node at (9,-3) {$-$};
		
		\draw[dashed] (11,-3) circle (1.414cm);
		\draw [thick,NavyBlue] (10,-2) -- (10.5,-2);
		\draw [thick, NavyBlue](10.5,-2)to[bend right=60](11.5,-2);
		\draw [thick, NavyBlue](10.5,-2)to[bend left=60](11.5,-2);
		\draw [thick,NavyBlue] (11.5,-2) -- (12,-2);
		
		\draw [thick,NavyBlue] (10,-4) -- (12,-4);
		
		\node at (13,-3) {$-$};
		
		\draw[dashed] (15,-3) circle (1.414cm);
		\draw [thick,NavyBlue] (14,-2) -- (16,-2);
		\draw [thick,NavyBlue] (14,-4) -- (14.5,-4);
		\draw[thick, NavyBlue](14.5,-4)to[bend right=60](15.5,-4);
		\draw[thick, NavyBlue](14.5,-4)to[bend left=60](15.5,-4);
		\draw [thick,NavyBlue] (15.5,-4) -- (16,-4);
		
		\node at (17,-3) {$)$};
	\end{scope}
	\end{tikzpicture}\tag{$IH^\text{mod}$}
\end{equation}
\noindent give the same elements in $\mathcal{G}\mathrm{raph}^\text{tri}[(\chi-2)^{-1}, (\chi-3)^{-1}, (\chi-4)^{-1}]$.
\end{theorem}

\begin{proof}
We establish this relation in $\mathcal{G}\mathrm{raph}^\text{tri}(\{a,b,c,d\}) \otimes \det \bQ^{\{a,b,c,d\}}$, and it then follows in general using functoriality on the signed Brauer category. We order the legs as $a < b < c < d$.

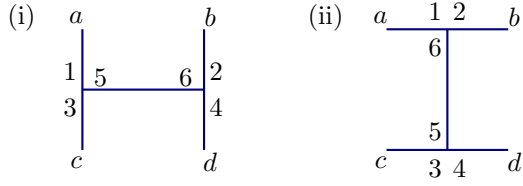
\begin{figure}[h]
	\begin{tikzpicture}[baseline=(current bounding box.center), arrowmark/.style 2 args={decoration={markings,mark=at position #1 with \arrow{#2}}}]
\begin{scope}[scale=0.8]
	
\node at (-1,1.2) {(i)};
	\node at (-0.1,1.2) {$a$};
	\node at (2.1,1.2) {$b$};
	\node at (-0.1,-1.2) {$c$};
	\node at (2.1,-1.2) {$d$};
	
	\node at (-0.2,0.3) {$1$};
	\node at (-0.2,-0.3) {$3$};
	\node at (0.3,0.2) {$5$};
	
	\node at (1.7,0.2) {$6$};
	\node at (2.2,0.3) {$2$};
	\node at (2.2,-0.3) {$4$};
	
		\draw [thick,NavyBlue] (0,0) -- (2,0);
		\draw [thick,NavyBlue] (0,-1) -- (0,1);
		\draw [thick,NavyBlue] (2,-1) -- (2,1);

\node at (4,1.2) {(ii)};
	\node at (4.9,1.2) {$a$};
	\node at (7.1,1.2) {$b$};
	\node at (4.9,-1.2) {$c$};
	\node at (7.1,-1.2) {$d$};
	
	\node at (5.8,1.3) {$1$};
	\node at (6.2,1.3) {$2$};
	\node at (5.8,0.7) {$6$};
	
	\node at (5.8,-1.3) {$3$};
	\node at (6.2,-1.3) {$4$};
	\node at (5.8,-0.7) {$5$};
	
		\draw [thick,NavyBlue] (5,1) -- (7,1);
		\draw [thick,NavyBlue] (6,-1) -- (6,1);
		\draw [thick,NavyBlue] (5,-1) -- (7,-1);
		
\end{scope}
	\end{tikzpicture}
\caption{Some marked graphs.}\label{fig:MarkedGphs}
\end{figure}

Consider first the $H$-shaped graph shown in Figure \ref{fig:MarkedGphs} (i), with the depicted names of half edges, ordered as $3<1<5<6<2<4$. Its corresponding permutation is
$\bigl(\begin{smallmatrix}
    3 & c & 1 & a & 5 & 6 & 2 & b & 4 & d \\
    3 & 1 & 5 & 6 &  2  & 4 & a & b &c & d
  \end{smallmatrix}\bigr)$ which is even. Thus this ordering data represents the underlying undecorated $H$-shaped trivalent graph. Ignoring for now the matchings to the legs (which are given by matching 1 with $a$, 2 with $b$, and so on), it corresponds to $\lambda_{5,6}(\kappa_{\bar{\epsilon}^{3, 1, 5}} \cdot \kappa_{\bar{\epsilon}^{6,2,4}})$. Using the form of the relations which avoid creating Euler classes from the proof of Theorem \ref{thm:AllAreTriv} we have
\begin{align*}
\lambda_{5,6}(\kappa_{\bar{\epsilon}^{3,1,5}} \cdot \kappa_{\bar{\epsilon}^{6,2,4}}) &= \kappa_{\bar{\epsilon}^{3,1,2,4}} + \tfrac{1}{\chi(\chi-2)} \kappa_{e^2} \kappa_{\bar{\epsilon}^{3,1}} \kappa_{\bar{\epsilon}^{2,4}}\\
&\quad - \tfrac{1}{\chi-2}(\lambda_{u,v}(\kappa_{\bar{\epsilon}^{u,v,3,1}}) \kappa_{\bar{\epsilon}^{2,4}} + \kappa_{\bar{\epsilon}^{3,1}} \lambda_{u,v}(\kappa_{\bar{\epsilon}^{u,v,2,4}})).
\end{align*}

Consider now the $I$-shaped graph shown in Figure \ref{fig:MarkedGphs} (ii), with the depicted names of the half-edges, ordered as $4 < 3 < 5 < 6< 1 < 2$. Its corresponding permutation is $\bigl(\begin{smallmatrix}
    4 & d & 3 & c & 5 & 6 & 1 & a & 2 & b \\
    4 & 3 & 5 & 6 &  1  & 2 & a & b &c & d
  \end{smallmatrix}\bigr)$ which is odd. Thus this ordering data represents \emph{minus} the underlying undecorated $I$-shaped trivalent graph. Ignoring again the matchings to the legs, it corresponds to
\begin{align*}
\lambda_{5,6}(\kappa_{\bar{\epsilon}^{4,3,5}} \cdot \kappa_{\bar{\epsilon}^{6,1,2}}) &= \kappa_{\bar{\epsilon}^{4,3,1,2}} + \tfrac{1}{\chi(\chi-2)} \kappa_{e^2} \kappa_{\bar{\epsilon}^{4,3}} \kappa_{\bar{\epsilon}^{1,2}}\\
&\quad - \tfrac{1}{\chi-2}(\lambda_{u,v}(\kappa_{\bar{\epsilon}^{u,v,4,3}}) \kappa_{\bar{\epsilon}^{1,2}} + \kappa_{\bar{\epsilon}^{4,3}} \lambda_{u,v}(\kappa_{\bar{\epsilon}^{u,v,1,2}})).
\end{align*}

The sum of these two expressions therefore represents the image under $\Phi$ of the difference $H-I$ of the underlying undecorated trivalent graphs. Furthermore, $\kappa_{\bar{\epsilon}^{4,3,1,2}} = - \kappa_{\bar{\epsilon}^{3,1,2,4}}$ so these terms cancel.

From the proof of Theorem \ref{thm:AllAreTriv} we have the identity
$$\lambda_{u,v}(\kappa_{\bar{\epsilon}^{u,v,s,t}}) = \tfrac{\chi-2}{\chi-4} \lambda_{i,j} \lambda_{k,l}(\kappa_{\bar{\epsilon}^{s,i,k}} \cdot \kappa_{\bar{\epsilon}^{l,j,t}}) - \tfrac{1}{\chi(\chi-4)} \kappa_{e^2}\kappa_{\bar{\epsilon}^{s,t}},$$
expressing terms of the form $\lambda_{u,v}(\kappa_{\bar{\epsilon}^{u,v,s,t}})$ in terms of (2- and) 3-valent vertices. Applying it to the sum of the two expressions above, and collecting terms, therefore gives
\begin{align*}
\Phi(H-I) &= \tfrac{1}{\chi(\chi-4)} \kappa_{e^2}\bigl(\kappa_{\bar{\epsilon}^{3,1}}\kappa_{\bar{\epsilon}^{2,4}} + \kappa_{\bar{\epsilon}^{4,3}}\kappa_{\bar{\epsilon}^{2,1}}\bigr)\\
&\quad - \tfrac{1}{\chi-4}\bigl(\lambda_{i,j}\lambda_{k,l}(\kappa_{\bar{\epsilon}^{3,i,k}} \cdot \kappa_{\bar{\epsilon}^{l,j,1}})\kappa_{\bar{\epsilon}^{2,4}} + \kappa_{\bar{\epsilon}^{3,1}}\lambda_{i,j}\lambda_{k,l}(\kappa_{\bar{\epsilon}^{2,i,k}} \cdot \kappa_{\bar{\epsilon}^{l,j,4}}) \\
&\quad\quad\quad\quad\quad \lambda_{i,j}\lambda_{k,l}(\kappa_{\bar{\epsilon}^{4,i,k}} \cdot \kappa_{\bar{\epsilon}^{l,j,3}})\kappa_{\bar{\epsilon}^{1,2}} + \kappa_{\bar{\epsilon}^{4,3}}\lambda_{i,j}\lambda_{k,l}(\kappa_{\bar{\epsilon}^{1,i,k}} \cdot \kappa_{\bar{\epsilon}^{l,j,2}}) \bigr).
\end{align*}
Using that $\kappa_{e^2} =\Phi(\tfrac{\chi}{\chi-3}\Theta)$ and carefully putting the graphs corresponding to the other terms into the normal form of Section \ref{sec:orderings} gives the identity in the statement of the theorem.
\end{proof}

Our relation $IH^\text{mod}$ is graphically identical to the relation called $IH_0^\text{bis}$ by Akazawa \cite[p.\ 100]{Akazawa} and in the corrigendum \cite{GNCorr} to the paper of Garoufalidis and Nakamura \cite{GN}. In those papers it is emphasised that $IH_0^\text{bis}$ means this identity is imposed only when the 4 half-edges belong to \emph{distinct} edges, but in fact this is redundant: if the 4 half-edges do not belong to distinct edges, then the identity already holds in $\mathcal{G}\mathrm{raph}^\text{undec}$. So in fact imposing our relation $IH^\text{mod}$ is identical to imposing their relation $IH_0^\text{bis}$.

\begin{theorem}\label{thm:TrivVsAllGraphs}
 Upon inverting $\chi-2$, $\chi-3$, and $\chi-4$, the maps 
$$\frac{\mathcal{G}\mathrm{raph}^\text{undec}(S)}{(IH^\text{mod})} \otimes \det \bQ^S \overset{\Phi}\lra \mathcal{G}\mathrm{raph}^\text{tri}(S) \overset{inc}\lra \mathcal{G}\mathrm{raph}(S)$$
are isomorphisms.
\end{theorem}

\begin{proof}
Let $R :=\bQ[\chi^{\pm 1}, (\chi-2)^{-1}, (\chi-3)^{-1}, (\chi-4)^{-1}]$ and implicitly base change to this ring. We have already shown in Corollary \ref{cor:TriIsAll} that the second map is an isomorphism, and $\Phi$ is certainly an epimorphism, so it remains to show that the composition is a monomorphism.

For an undecorated trivalent graph $\Gamma$, define a \emph{double edge} to be an unordered pair of vertices which share precisely two edges, and a \emph{triple edge} to be an unordered pair of vertices which share precisely three edges, i.e.\ form a theta-graph. Define
$$\mu(\Gamma) := 2 \cdot \# \text{double edges of $\Gamma$} + 3 \cdot \# \text{triple edges of $\Gamma$},$$
 filter $\mathcal{G}\mathrm{raph}^\text{undec}$ by letting $F^k \mathcal{G}\mathrm{raph}^\text{undec}$ be spanned by those $\Gamma$ with $\mu(\Gamma) \geq k$, and give $\mathcal{G}\mathrm{raph}^\text{undec}/(IH^\text{mod})$ the induced filtration.

If $\Gamma = \Gamma_H$ is a graph with $\mu(\Gamma)=k$ and a distinguished ``$H$" subgraph, and $\Gamma_I$ is obtained by replacing this ``$H$''-subgraph by ``$I$'', then by applying the relation $IH^\text{mod}$ to this subgraph we find that
\begin{enumerate}[(i)]
\item if the edge involved is not part of a double or triple edge then the relation gives $\Gamma_H - \Gamma_I \in F^{k+1} \mathcal{G}\mathrm{raph}^\text{undec}/(IH^\text{mod})$,
\item if the edge involved is part of a double or triple edge then the relation is trivial (i.e.\ already holds in $\mathcal{G}\mathrm{raph}^\text{undec}$).
\end{enumerate}
Thus the associated graded of the induced filtration on $\mathcal{G}\mathrm{raph}^\text{undec}/(IH^\text{mod})$ can be described as $\mathcal{G}\mathrm{raph}^\text{undec}/(IH_0)$, where as in \cite{GN} the relation $IH_0$ means imposing the ``$I=H$'' relation when the four half-edges belong to different edges. Now $IH_0$ is an equivalence relation on the set of isomorphism classes of trivalent graphs without loops, and similarly to \cite[Proof of Proposition 2.3 (c)]{GN} it is easy to see that all connected trivalent graph without loops of the same rank and with the same legs are equivalent to each other: in other words the equivalences classes of such are given by partitions of $S$ (the parts are the legs of each connected component) labelled by a power of $e$ (recording the rank of the graph). It follows that the rank of $\mathcal{G}\mathrm{raph}^\text{undec}/(IH^\text{mod})$ in each degree, as an $R$-module, is at most that of $\mathcal{G}\mathrm{raph}(\emptyset)$ as determined in Lemma \ref{lem:JustGammas}, and so the composition in the statement of the theorem, which is an epimorphism, must be an isomorphism.
\end{proof}

\subsection{On the work of Garoufalidis and Nakamura}\label{sec:GN}

The discussion of the last few sections can be used to complete the work of Garoufalidis and Nakamura \cite{GN, GNCorr}, concerning the calculation of the invariants $[\Lambda^* V_{1^3} / (V_{2^2})]^{\mathrm{Sp}}$ in a stable range. Here we write $V_\lambda$ for the irreducible $\mathrm{Sp}$-representation corresponding to the partition $\lambda$, which was written as $[\lambda]_{\mathfrak{sp}}$ in those papers, and $V_{2^2}$ denotes the unique copy of this irreducible in $\Lambda^2 V_{1^3}$. Combining Theorem 1.1 and Proposition 2.3 (c) of \cite{GN} was supposed to calculate $[\Lambda^* V_{1^3} / (V_{2^2})]^{\mathrm{Sp}}$ in a stable range, but for the corrected version of Theorem 1.1 in \cite{GNCorr}, which expresses these invariants as $\mathcal{G}\mathrm{raph}^\text{undec}(\emptyset)_g/(IH_0^{bis})$, the authors say ``it turns out that a simple stable structure of [these invariants] as in Proposition 2.3 (c) will not be easy to detect''. However Theorem \ref{thm:TrivVsAllGraphs} and equation \eqref{eq:CalcGraphsEmpty} gives that
$$[\Lambda^* V_{1^3} / (V_{2^2})]^{\mathrm{Sp}} \cong \mathcal{G}\mathrm{raph}^\text{undec}(\emptyset)_g/(IH_0^{bis}) \cong \mathcal{G}\mathrm{raph}(\emptyset)_g \cong \bQ[\Gamma_1, \Gamma_2, \ldots]$$
in a stable range. Thus in fact Proposition 2.3 (c) of \cite{GN} is correct as stated.

\begin{remark}
This can also be obtained from the work of Felder, Naef, and Willwacher \cite{FNW}. Specifically, the graded-commutative algebra $A_{(g)}$ defined just before Theorem 6 of that paper is $\Lambda^* V_{1^3} / (V_{2^2})$, and Theorem 6 together with Proposition 36 (3) also gives the above.
\end{remark}

\bibliographystyle{amsalpha}
\bibliography{biblio}

\end{document}